\documentclass[11pt,dvipsnames,svgnames,table]{article}
\usepackage[T1]{fontenc}
\usepackage{lmodern,amsmath,amsthm,amsfonts,amssymb,graphicx,float,microtype,thmtools,underscore,mathtools,anyfontsize,tikz} 
\usepackage{xcolor} 
\usepackage{subcaption}
\usepackage{hyperref}
\hypersetup{
colorlinks,
linkcolor={black},
citecolor={black},
urlcolor={blue!60!black},
pdftitle={Clustered Colouring of Graph Products},
pdfauthor={Rutger~Campbell, J. Pascal Gollin, Kevin~Hendrey, Thomas Lesgourgues, Bojan Mohar, Youri Tamitegama, Jane Tan, David R. Wood}} 
\usepackage{multirow}
\usepackage[noabbrev,capitalise]{cleveref}
\crefname{lem}{Lemma}{Lemmas}
\crefname{thm}{Theorem}{Theorems}
\crefname{cor}{Corollary}{Corollaries}
\crefname{prop}{Proposition}{Propositions}
\crefname{claim}{Claim}{Claims}
\crefname{conj}{Conjecture}{Conjectures}
\crefname{openproblem}{Open Problem}{Open Problems}
\crefformat{equation}{(#2#1#3)}
\Crefformat{equation}{Equation #2(#1)#3}
\usepackage[shortlabels]{enumitem}
\setlist[itemize]{topsep=0ex,itemsep=0ex,parsep=0.4ex}
\setlist[enumerate]{topsep=0ex,itemsep=0ex,parsep=0.4ex}
\newcommand{\defn}[1]{\textcolor{Maroon}{\emph{#1}}}
\newcommand{\CartProd}{\mathbin{\square}}

\newcommand{\NN}{\mathbb{N}}

\newcommand{\WW}{\mathcal{W}}
\usepackage[longnamesfirst,numbers,sort&compress]{natbib}
\makeatletter
\def\NAT@spacechar{~}
\makeatother
\usepackage[bmargin=26mm,tmargin=26mm,lmargin=32mm,rmargin=32mm]{geometry}
\allowdisplaybreaks
\newcommand{\half}{\ensuremath{\protect\tfrac{1}{2}}}
\DeclarePairedDelimiter{\floor}{\lfloor}{\rfloor}

\DeclarePairedDelimiter{\parens}{(}{)}
\DeclarePairedDelimiter{\set}{\{}{\}}

\DeclarePairedDelimiter\abs{\lvert}{\rvert}   
 
\DeclarePairedDelimiter\size{\lvert}{\rvert}

\renewcommand{\geq}{\geqslant}
\renewcommand{\leq}{\leqslant}

\DeclareMathOperator{\tw}{tw}

\renewcommand{\thefootnote}{\fnsymbol{footnote}}
\allowdisplaybreaks
\theoremstyle{plain}
\newtheorem{thm}{Theorem}
\newtheorem{lem}[thm]{Lemma}
\newtheorem{claim}{Claim}[thm]

\newtheorem{prop}[thm]{Proposition}

\theoremstyle{definition}
\newtheorem{conj}[thm]{Conjecture}
\usepackage{scalerel,stackengine}
\stackMath
\newcommand\cone[1]{%
\savestack{\tmpbox}{\stretchto{%
  \scaleto{%
    \scalerel*[\widthof{\ensuremath{#1}}]{\kern-.6pt\bigwedge\kern-.6pt}%
    {\rule[-\textheight/2]{1ex}{\textheight}}
  }{\textheight}%
}{0.5ex}}%
\stackon[1pt]{#1}{\tmpbox}%
}

\usepackage{ifthen}

\usepackage{tikz}
\usepackage{calc}
\usetikzlibrary{calc,patterns,decorations.pathreplacing,backgrounds,shapes}

\tikzstyle{vertex}=[circle, draw, fill=black, inner sep=0pt, minimum size=5pt] 
\tikzstyle{smallvertex}=[circle, draw, fill=black, inner sep=0pt, minimum size=4pt] 
\tikzstyle{tinyvertex}=[circle, draw, fill=black, inner sep=0pt, minimum size=3pt] 
\newcommand{\vertex}{\node[vertex]} 
\newcommand{\smallvertex}{\node[smallvertex]} 
\newcommand{\tinyvertex}{\node[tinyvertex]}

\tikzset
{%
  pics/onefan/.style args={#1--#2}{%
    code={%
    
    \vertex (top) at (3,1) [label=90:]{};
    
    \foreach \i in {1, 2, ..., 5} {
	    \vertex (c\i) at (\i,0) [label=-90:]{};}

	\foreach \i in {1, 2, ..., 5} {
	    \path (c\i) edge (top);}
	
    \begin{scope}[on background layer]
        \path (c1) edge (c4);
        \path (c4) edge (c5);
    \end{scope}
	
	\pgfmathsetmacro\raiseshift{#2/6}
	\draw [decorate,decoration={brace,mirror,raise=\raiseshift cm}] (c1.west) -- (c5.east) 
    node [pos=0.5,anchor=north,yshift=-\raiseshift cm]  {$n^4$}; 
}}}

\newcommand{\FanBrace}[1]
{
\pgfmathsetmacro\y{#1*4}
\begin{tikzpicture}[x=#1 cm, y=\y cm]

    \vertex (dom) at (10,0) [label=90:$x$]{};
    
    \foreach \j in {0, 5, 10, 15}{
    \begin{scope}[shift={(\j,-2)}]
        \path (dom) edge (3,1);
        \pic at (0,0) {onefan=#1--\y};
        \foreach \i in {1, ..., 5}{
            \path (dom) edge[color=gray, line width = .5, dotted] (\i,0);
        }
    \end{scope}}
    
	\pgfmathsetmacro\raiseshift{\y/2.5}
	\draw [decorate,decoration={brace,mirror,raise=\raiseshift cm}] (0.75,-2) -- (20.25,-2) 
    node [pos=0.5,anchor=north,yshift=-\raiseshift cm]  {\Large $n^2$}; 	

\end{tikzpicture}}

\tikzset
{%
  pics/onegrid/.style={%
    code={%
    
    \tinyvertex (a) at (0,0) [label=]{};
    \tinyvertex (b) at (1,0) [label=]{};
    \tinyvertex (c) at (1,1) [label=]{};
    \tinyvertex (d) at (0,1) [label=]{};

    \draw[gray]
    (a) -- (b) -- (c) -- (d) -- (a) -- (c)
    (b) -- (d);
    }},
}

\newcommand{\triggrid}[3]{
    
\begin{tikzpicture}[x=#2 cm, y=#2 cm]

    \pgfmathsetmacro\size{int(#1-1)}

    \foreach \i in {0, 1, ..., \size} {
    \foreach \j in {0, 1, ..., \size} {
        \pic (grid\i\j) at (\i,\j) {onegrid} ;
    }}
    
    \pgfmathsetmacro\y{0}
    \pgfmathsetmacro\z{\size}    
    \foreach \x in {0, 1, ..., \size} {
        \path
        (grid\x\y a) edge[line width = 1.5, orange] (grid\x\y b)
        (grid\x\z d) edge[line width = 1.5, orange] (grid\x\z c);}
    
    \pgfmathsetmacro\x{\size/2+0.5}
    \pgfmathsetmacro\y{\size+1.5}
    \draw[orange] (\x,-0.5) node{$R_s$};
    \draw[orange] (\x,\y) node{$R_t$};
    
    \pgfmathsetmacro\x{0}
    \pgfmathsetmacro\z{\size}    
    \foreach \y in {0, 1, ..., \size} {
        \path
        (grid\x\y a) edge[line width = 1.5, violet] (grid\x\y d)
        (grid\z\y b) edge[line width = 1.5, violet] (grid\z\y c);}
    
    \pgfmathsetmacro\y{\size/2+0.5}
    \pgfmathsetmacro\x{\size+1.5}
    \draw[violet] (-0.5,\y) node{$C$};    
    \draw[violet] (\x,\y) node{$C'$};

    \ifthenelse{#3=1}{
    
    \pgfmathsetmacro\y{#1/2}
    \pgfmathsetmacro\z{3*#1/4}
    \pgfmathsetmacro\x{#1+\z}
    \smallvertex[red] (v) at (-\z,\y) [label={[red]180:$v$}]{};
    \smallvertex[ForestGreen] (vprime) at (\x,\y) [label={[ForestGreen]0:$v'$}]{};
    \smallvertex[blue] (us) at (\y,\x) [label={[blue]90:$u_s$}]{};
    \smallvertex[blue] (ut) at (\y,-\z) [label={[blue]-90:$u_t$}]{};
    
    \foreach \y in {0, 1, ..., #1} {
    \path
    (v) edge[color=gray] (0,\y)
    (vprime) edge[color=gray] (#1,\y)
    (us) edge[color=gray] (\y,#1)
    (ut) edge[color=gray] (\y,0);
    }
    
    }{}
    
  \end{tikzpicture}  
}

\tikzstyle{BlackVertex}=[circle, draw, fill=black, inner sep=0pt, minimum size=6pt] 
\tikzstyle{BlueVertex}=[circle, draw=blue, fill=blue, inner sep=0pt, minimum size=6pt] 
\tikzstyle{RedVertex}=[circle, draw=red, fill=red, inner sep=0pt, minimum size=6pt] 
\tikzstyle{GreenVertex}=[circle, draw=teal, fill=teal, inner sep=0pt, minimum size=6pt] 
\newcommand{\BlackVertex}{\node[BlackVertex]} 
\newcommand{\BlueVertex}{\node[BlueVertex]} 
\newcommand{\RedVertex}{\node[RedVertex]} 
\newcommand{\GreenVertex}{\node[GreenVertex]}

\newcommand{\FanFan}[1]{

\begin{tikzpicture}[x=#1 cm, y=#1 cm]
\pgfmathsetmacro\n{22}
\begin{scope}[rotate=-90]

\begin{scope}[on background layer]
\foreach \i [evaluate={\ii=int(\i+1);}] in {1, 2,...,\n}{
    \foreach \j [evaluate={\jj=int(\j+1);}] in {1, 2,...,\n}{
        \draw [help lines] (\i,\j) -- (\i,\jj);
        \draw [help lines] (\i,\j) -- (\ii,\jj);
        \draw [help lines] (\ii,\j) -- (\i,\jj);   
        \draw [help lines] (\i,\j) -- (\ii,\j);                
    }
}
\draw [help lines] (\n+1,1) -- (\n+1,\n+1);
\draw [help lines] (1,\n+1) -- (\n+1,\n+1);

\foreach \i [evaluate={\ii=int(\i+1);}] in {0,...,\n}{
    \foreach \j [evaluate={\jj=int(\j+1);}] in {0,...,\n}{
        \GreenVertex[color=teal] at (\ii,\jj) [label=-90:]{};
    }}
\end{scope}

\BlackVertex (orig) at (0,0) [label=135:\Large ${(0,0)}$]{};

\path (0,1) edge[color=red, opacity = 1, line width = 3,dotted] (0,3);
\draw [thick, <->,color=red] (-0.5,1) -- (-0.5,3) node[midway,above] {\Large $n$};
\RedVertex at (0,1) [label=-90:]{};
\RedVertex at (0,3) [label=-90:]{};

\BlackVertex at (0,4) [label=-90:]{};
\BlackVertex at (0,5) [label=-90:]{};
\BlackVertex at (0,6) [label=-90:]{};

\RedVertex at (0,7) [label=-90:]{};
\RedVertex at (0,9) [label=-90:]{};
\path (0,7) edge[color=red, opacity = 1, line width = 3,dotted] (0,9);
\draw [thick,<->,color=red] (-0.5,7) -- (-0.5,9) node[midway,above] {\Large $n$};

\BlackVertex at (0,10) [label=-90:]{};
\BlackVertex at (0,11) [label=-90:]{};
\BlackVertex at (0,12) [label=-90:]{};

\path (0,13) edge[color=red, opacity = 1, line width = 3,dotted] (0,15);
\draw [thick, <->,color=red] (-0.5,13) -- (-0.5,15) node[midway,above] {\Large $n$};
\RedVertex at (0,13) [label=-90:]{};
\RedVertex at (0,15) [label=-90:]{};

\BlackVertex at (0,16) [label=-90:]{};
\BlackVertex at (0,17) [label=-90:]{};
\BlackVertex at (0,18) [label=-90:]{};

\path (0,19) edge[color=red, opacity = 1, line width = 3,dotted] (0,21);
\draw [thick, <->,color=red] (-0.5,19) -- (-0.5,21) node[midway,above] {\Large $n$};
\RedVertex at (0,19) [label=-90:]{};
\RedVertex at (0,21) [label=-90:]{};

\BlackVertex at (0,22) [label=-90:]{};
\BlackVertex at (0,23) [label=-90:]{};
\end{scope}

\begin{scope}[rotate=180]
\path (0,1) edge[color=blue, opacity = 1, line width = 3,dotted] (0,3);
\draw [thick, <->,color=blue] (0.5,1) -- (0.5,3) node[midway,above,rotate=90] {\Large $n$};
\BlueVertex at (0,1) [label=-90:]{};
\BlueVertex at (0,3) [label=-90:]{};

\BlackVertex at (0,4) [label=-90:]{};
\BlackVertex at (0,5) [label=-90:]{};
\BlackVertex at (0,6) [label=-90:]{};

\BlueVertex at (0,7) [label=-90:]{};
\BlueVertex at (0,9) [label=-90:]{};
\path (0,7) edge[color=blue, opacity = 1, line width = 3,dotted] (0,9);
\draw [thick, <->,color=blue] (0.5,7) -- (0.5,9) node[midway,above,rotate=90] {\Large $n$};

\BlackVertex at (0,10) [label=-90:]{};
\BlackVertex at (0,11) [label=-90:]{};
\BlackVertex at (0,12) [label=-90:]{};

\path (0,13) edge[color=blue, opacity = 1, line width = 3,dotted] (0,15);
\draw [thick, <->,color=blue] (0.5,13) -- (0.5,15) node[midway,above,rotate=90] {\Large $n$};
\BlueVertex at (0,13) [label=-90:]{};
\BlueVertex at (0,15) [label=-90:]{};

\BlackVertex at (0,16) [label=-90:]{};
\BlackVertex at (0,17) [label=-90:]{};
\BlackVertex at (0,18) [label=-90:]{};

\path (0,19) edge[color=blue, opacity = 1, line width = 3,dotted] (0,21);
\draw [thick, <->,color=blue] (0.5,19) -- (0.5,21) node[midway,above,rotate=90] {\Large $n$};
\BlueVertex at (0,19) [label=-90:]{};
\BlueVertex at (0,21) [label=-90:]{};

\BlackVertex at (0,22) [label=-90:]{};
\BlackVertex at (0,23) [label=-90:]{};
\end{scope}

\begin{scope}[rotate=-90]
\foreach \j in {5, 11, 17, 23}{
    \foreach \i [evaluate={\ii=int(\i+1);}] in {0, ...,\n}{
        \RedVertex[color=red] at (\ii,\j) [label=-90:]{};
    }
}

\foreach \i in {5, 11, 17, 23}{
    \foreach \j [evaluate={\jj=int(\j+1);}] in {0, ...,\n}{
        \BlueVertex[color=blue] at (\i,\jj) [label=-90:]{};
    }}

\RedVertex[color=red] at (5,5) [label=-90:]{};
\RedVertex[color=red] at (5,17) [label=-90:]{};
\RedVertex[color=red] at (11,11) [label=-90:]{};
\RedVertex[color=red] at (11,23) [label=-90:]{};
\RedVertex[color=red] at (17,5) [label=-90:]{};
\RedVertex[color=red] at (17,17) [label=-90:]{};
\RedVertex[color=red] at (23,11) [label=-90:]{};
\RedVertex[color=red] at (23,23) [label=-90:]{};

\end{scope}

\begin{scope}[on background layer, rotate = -90]
\draw[fill=gray!25, draw=gray!25] (0.5,0.75) rectangle ++(\n+1,-2.5);
\draw[fill=gray!25, draw=gray!25] (0.5,0.75) rectangle ++(-2.5,\n+1);
\draw[gray] (-1,11) node{\Huge $Y$};
\draw[gray] (11,-1) node{\Huge $X$};
\end{scope}

\begin{scope}[rotate=-90]
\foreach \i in {1,6, 12, 18}{
    \foreach \j in {1,6, 12, 18}{
        \ifnum\i=1
            \ifnum\j=1
                \draw[fill=teal, draw=teal,opacity=0.75] (\i,\j) rectangle ++(3,3);
             \fi
            \ifnum\j>1
                \draw[fill=teal, draw=teal,opacity=0.75] (\i,\j) rectangle ++(3,4);
             \fi             
        \fi
        \ifnum\i>1
            \ifnum\j=1
                \draw[fill=teal, draw=teal,opacity=0.75] (\i,\j) rectangle ++(4,3);
             \fi
            \ifnum\j>1
                \draw[fill=teal, draw=teal,opacity=0.75] (\i,\j) rectangle ++(4,4);
             \fi             
        \fi        
    }}
\end{scope}

\end{tikzpicture}

}

\begin{document}
	
\author{
Rutger~Campbell\,\footnotemark[2] \qquad 
J.~Pascal Gollin\,\footnotemark[2] \qquad
Kevin~Hendrey\,\footnotemark[2] \\
Thomas Lesgourgues\,\footnotemark[3] \qquad 
Bojan Mohar\,\footnotemark[5] \\
Youri Tamitegama\,\footnotemark[4] \qquad
Jane Tan\,\footnotemark[4] \qquad 
David~R.~Wood\,\footnotemark[6]
}

\footnotetext[2]{Discrete Mathematics Group, Institute for Basic Science (IBS), Daejeon, Republic~of~Korea (\texttt{\{rutger,kevinhendrey,pascalgollin\}@ibs.re.kr}). Research of R.C.\ and K.H.\ 
supported by the Institute for Basic Science (IBS-R029-C1). Research of J.P.G.\ supported by the Institute for Basic Science (IBS-R029-Y3).}
\footnotetext[3]{Department of Combinatorics and Optimization, University of Waterloo, Canada (\texttt{tlesgourgues@uwaterloo.ca}).}
\footnotetext[4]{Mathematical Institute, University of Oxford, UK (\texttt{\{tamitegama,jane.tan\}@maths.ox.ac.uk}). }
\footnotetext[5]{Department of Mathematics, Simon Fraser University, Burnaby, BC, Canada (\texttt{mohar@sfu.ca}). Supported in part by the NSERC Doscovery Grant R832714 (Canada) and
by the ERC Synergy grant (European Union, ERC, KARST, project number 101071836).}
\footnotetext[6]{School of Mathematics, Monash   University, Melbourne, Australia  (\texttt{david.wood@monash.edu}). Research supported by the Australian Research Council.}

\sloppy
	
\title{\textbf{Clustered Colouring of Graph Products}}
	
\maketitle
	
\begin{abstract}
A colouring of a graph $G$ has \defn{clustering} $k$ if the maximum number of vertices in a monochromatic component equals $k$. Motivated by recent results showing that many natural graph classes are subgraphs of the strong product of a graph with bounded treewidth and a path, this paper studies clustered colouring of strong products of two bounded treewidth graphs, where none, one, or both graphs have bounded degree. For example, in the case of two colours, if $n$ is the number of vertices in the product, then we show that clustering $\Theta(n^{2/3})$ is best possible, even if one of the graphs is a path. However, if both graphs have bounded degree, then clustering $\Theta(n^{1/2})$ is best possible. With three colours, if one of the graphs has bounded degree, then we show that clustering $\Theta(n^{3/7})$ is best possible. However, if neither graph has bounded degree, then clustering $\Omega(n^{1/2})$ is necessary. More general bounds for any given number of colours are also presented.
\end{abstract}
	
\renewcommand{\thefootnote}{\arabic{footnote}}


\section{Introduction}

A \defn{colouring} of a graph $G$ is a function that assigns a `colour' to each vertex of $G$. A \defn{$c$-colouring} is a colouring that uses at most $c$ colours. We allow adjacent vertices to be assigned the same colour. A colouring of a graph has \defn{clustering} $k$ if the maximum number of vertices in a monochromatic component equals $k$. Here, a \defn{monochromatic component} is a component of the subgraph induced by the vertices assigned a single colour. For example, a graph with chromatic number $c$ is $c$-colourable with clustering $1$. See \citep{WoodSurvey} for an extensive survey on this topic.

Most research on clustered colouring gives constant bounds on the clustering (independent of $|V(G)|$). Some other papers focus on the case where the number of colours is so small that dependence on $|V(G)|$ is unavoidable (see \citep{LMST08,MP08,BBK19} for example). We take the latter approach. In particular, we focus on clustered colouring of strong products of two graphs with bounded treewidth (see~\cref{sec:definitions} for definitions of these notions). A strong motivation for considering such graphs is that it has been recently shown that many natural graph classes are subgraphs of the strong product of a graph with bounded treewidth and a path~\cite{DJMMUW20,DMW23,DHHW22,ISW,HW24} (see \cref{Planar} for more on this theme).  

We study two variations of these products. In the first, we assume that one of the graphs has bounded maximum degree, generalising the paths in the aforementioned products. We then study the scenario with no restriction on the maximum degree. In each scenario, given an integer $c\geq 2$, our goal is to prove tight bounds on the optimal clustering in $c$-colourings of products $H_1\boxtimes H_2$, where $H_1$ and $H_2$ are graphs of bounded treewidth. Our main results are stated in the next section. They focus on the exponent of $|V(H_1\boxtimes H_2)|$ in the clustering value. More precise statements accompany their proofs.

\subsection{Main results}

First, consider $2$-colourings of graph products. We show the following asymptotically tight bound on the optimal clustering in this case. 

\begin{thm}
\label{thm:TwoColoursUpperAndLower}
For any fixed positive integer $t$ and any graphs $H_1,H_2$ with treewidth at most~$t$, the graph $H_1\boxtimes H_2$ is $2$-colourable with clustering $O(|V(H_1\boxtimes H_2)|^{2/3})$. Furthermore, there are infinitely many graphs $H$ with treewidth 2 and paths $P$ such that every 2-colouring of $H\boxtimes P$ has clustering $\Omega(|V(H\boxtimes P)|^{2/3})$.
\end{thm}

\cref{thm:TwoColoursUpperAndLower} shows that in the case of two colours, perhaps surprisingly, the maximum degree condition has no significant impact on the optimal clustering value. For more than two colours, the maximum degree condition has a significant impact on the optimal clustering value of colourings of graph products. 

Now consider the case of three colours. When one graph has bounded degree, we prove the following asymptotically tight bound.

\begin{thm}\label{thm:ThreeColoursMaxDegreeUpperAndLower}
For any fixed positive integers $t,\Delta$, for any graph $H_1$ with treewidth at most~$t$, and any graph $H_2$ with treewidth at most $t$ and maximum degree at most $\Delta$, the graph $H_1\boxtimes H_2$ is $3$-colourable with clustering $O(|V(H_1\boxtimes H_2)|^{3/7})$. Furthermore, there are infinitely many graphs $H$ with treewidth 3 and paths $P$ such that every 3-colouring of $H\boxtimes P$ has clustering $\Omega(|V(H\boxtimes P)|^{3/7})$.
\end{thm}

The following result shows that dropping the maximum degree restriction impacts the exponent in both upper and lower bounds.

\begin{thm}\label{thm:ThreeColoursUpperAndLower}
For any fixed positive integer $t$ and any graphs $H_1,H_2$ with treewidth at most~$t$, the graph $H_1\boxtimes H_2$ is $3$-colourable with clustering $O(|V(H_1\boxtimes H_2)|^{4/7})$. Furthermore, there exist infinitely many pairs of graphs $H_1,H_2$ with treewidth $2$ such that every $3$-colouring of $H_1\boxtimes H_2$ has clustering $\Omega(|V(H_1\boxtimes H_2)|^{1/2})$.
\end{thm}

For $c\geq 4$ colours, when one graph has bounded degree, we obtain the following general upper bound. 

\begin{thm}\label{thm:UpperBoundWithMaxDegree}
For any fixed positive integers $t,c,\Delta$, for any graph $H_1$ with treewidth at most~$t$, and any graph $H_2$ with treewidth at most $t$ and maximum degree at most $\Delta$, the graph $H_1\boxtimes H_2$ is $c$-colourable with clustering $O(|V(H_1\boxtimes H_2)|^{c/(c^2-c+1)})$. 
\end{thm}

We use the lower bound from~\cref{thm:ThreeColoursMaxDegreeUpperAndLower} to prove the following general lower bound.

\begin{prop}\label{thm:GeneralLowerBoundMaxDegree}
For any fixed integer $c\geq 3$, there exist infinitely many graphs $H$ with treewidth $c$ and paths $P$, such that every $c$-colouring of $H\boxtimes P$ has clustering $\Omega(|V(H\boxtimes P)|^{1/(c-\frac{2}{3})})$.
\end{prop}

There is a sharp change in the asymptotic behaviour of both upper and lower bounds on the clustering in $c$-colourings of $H_1\boxtimes H_2$ when we remove the restriction on the maximum degree of $H_2$. Indeed, while the bounds of~\cref{thm:UpperBoundWithMaxDegree,thm:GeneralLowerBoundMaxDegree} are of the form $n^{1/(c-O(1))}$, in the unrestricted setting we prove the following upper and lower bound of the form $n^{\Theta(1/\sqrt{c})}$. 

\begin{thm}\label{thm:GeneralNoMaxDegree}
For any fixed positive integers $t,c$ with $c\geq 4$, for any graphs $H_1$ and $H_2$ with treewidth at most $t$, the graph $H_1\boxtimes H_2$ is $c$-colourable with clustering $O(|V(H_1\boxtimes H_2)|^{1/\floor{\sqrt{c}}})$. Furthermore, for every integer $c\geq 4$, there are infinitely many pairs of graphs $H_1,H_2$ with treewidth at most $\sqrt{c}$ such that every $c$-colouring of $H_1\boxtimes H_2$ has clustering $\Omega(|V(H_1\boxtimes H_2)|^{1/2\sqrt{c}})$.
\end{thm}

Of course,~\cref{thm:UpperBoundWithMaxDegree,thm:GeneralNoMaxDegree} are only of interest for $c<(t+1)^2$. Indeed, graphs with treewidth $t$ are $t$-degenerate and thus properly $(t+1)$-colourable; that is, they are $(t+1)$-colourable with clustering $1$. A product colouring then shows that $H_1\boxtimes H_2$ is properly $(t+1)^2$-colourable. 

\subsection{Summary}

The paper is organised as follows. In~\cref{sec:definitions} we establish the basic notation and key tools. Our result with two colours, 
\cref{thm:TwoColoursUpperAndLower}, is proved in \cref{sec:TwoColours}. The results with three colours, \cref{thm:ThreeColoursMaxDegreeUpperAndLower,thm:ThreeColoursUpperAndLower}, are proved in \cref{sec:ThreeColours}. Our most general results,~\cref{thm:UpperBoundWithMaxDegree} together with its associated lower bound given in~\cref{thm:GeneralLowerBoundMaxDegree} as well as~\cref{thm:GeneralNoMaxDegree}, are proved in~\cref{sec:ArbitraryNbColours}. \cref{sec:BothBoundedDegree} studies clustered colouring of products of two graphs, both with bounded treewidth and bounded maximum degree. This setting turns out to be much simpler than the above two scenarios, and the proofs are straightforward. \cref{Planar} concludes with a discussion on planar graphs.

\cref{tab:summary} summarises all our results. The notation $\Omega(f(n))$ means that for some $\delta=\delta(t,\Delta)>0$, there exist infinitely many graphs $H_1,H_2$ such that every $c$-colouring of $H_1\boxtimes H_2$ has clustering at least $\delta\, f(|V(H_1\boxtimes H_2)|)$. The notation $O(f(n))$ means that for some $\delta=\delta(t,\Delta)>0$ and for any graphs $H_1,H_2$, there exists a $c$-colouring of $H_1\boxtimes H_2$ with clustering at most $\delta\, f(|V(H_1\boxtimes H_2)|)$. If $f(n)$ does not depend on $n$ (that is, $c\geq 3$ in the first column), the constant $\delta$ is absolute (and does not depend on $t,\Delta$).

 \begin{table}[ht]
\caption{Clustering in $c$-colourings of $H_1\boxtimes H_2$, where $n:=|V(H_1\boxtimes H_2)|$.}
\label{tab:summary}
\centering
\begin{tabular}{l|c|c|c}
\hline
        & $\Delta(H_1)\leq\Delta$, $\tw(H_1)\leq t$& $\tw(H_1)\leq t$ & $\tw(H_1)\leq t$ \\
        & $\Delta(H_2)\leq\Delta$, $\tw(H_2)\leq t$&$\Delta(H_2)\leq\Delta$, $\tw(H_2)\leq t$&$\tw(H_2)\leq t$\\
        \hline
        $c=2$&$\Theta(n^{1/2})$&$\Theta(n^{2/3})$&$\Theta(n^{2/3})$\\
        $c=3$&$O(t^3\Delta^4)$&
        $\Theta(n^{3/7})$&$\Omega(n^{1/2}) \;\; O(n^{4/7})$\\
        $c\geq4$&$O(t^2\Delta^2)$&$\Omega(n^{1/(c-2/3)}) \;\; O(n^{c/(c^2-c+1)})$&$\Omega(n^{1/(2\sqrt{c})}) \;\; O(n^{1/\floor{\sqrt{c}}})$\\
           \hline
		     \end{tabular}
	 \end{table}

For each case in which upper and lower bounds do not match, we conjecture the recorded upper bound to be asymptotically tight. In particular, we believe the upper bound in~\cref{thm:UpperBoundWithMaxDegree} is asymptotically tight, and furthermore conjecture that there exists a construction matching this upper bound where the bounded degree graph is a path.

\begin{conj}\label{conj:GeneralLowerBoundMaxDegree}
For any fixed integer $c\geq 2$, there exists an integer $t$ such that for infinitely many graphs $H$ with treewidth $t$ and paths $P$, every $c$-colouring of $H\boxtimes P$ has clustering $\Omega(|V(H\boxtimes P)|^{c/(c^2-c+1)})$.
\end{conj}

We also believe that the upper bounds in~\cref{thm:GeneralNoMaxDegree,thm:ThreeColoursUpperAndLower} are asymptotically tight.

\begin{conj}\label{conj:LowerBoundThreeUnrestricted}
There exists a positive integer $t$ such that for infinitely many graphs $H_1,H_2$ with treewidth $t$, every $3$-colouring of $H_1\boxtimes H_2$ has clustering $\Omega(|V(H_1\boxtimes H_2)|^{4/7})$.
\end{conj}

\begin{conj}\label{conj:GeneralLowerBoundUnrestricted}
For any fixed integer $c\geq 4$ there exists an integer $k$ such that for infinitely many graphs $H_1,H_2$ with treewidth $k$, every $c$-colouring of $H_1\boxtimes H_2$ has clustering $\Omega(|V(H_1\boxtimes H_2)|^{1/((1+o_c(1))\sqrt{c})})$.
\end{conj}


\section{Definitions and tools}\label{sec:definitions}

We use standard graph-theoretic notation. In particular, given a graph $G$, we write $V(G)$ for its vertex set and $E(G)$ for its edge set.

For a tree~$T$, a \defn{$T$-decomposition} of a graph~$G$ is a collection~${\WW = (W_x \colon x \in V(T))}$ of subsets of~${V(G)}$ indexed by the nodes of~$T$ such that:
\begin{enumerate}[label=(\roman*)]
    \item for every edge~${vw \in E(G)}$, there exists a node~${x \in V(T)}$ with~${v,w \in W_x}$; and
    \item for every vertex~${v \in V(G)}$, the set~${\{ x \in V(T) \colon v \in W_x \}}$ induces a (connected) subtree of~$T$. 
\end{enumerate}
Each set~$W_x$ in~$\WW$ is called a \defn{bag}. 
The \defn{width} of~$\WW$ is~${\max\{ \abs{W_x} \colon x \in V(T) \}-1}$. 
A \defn{tree-decomposition} is a $T$-decomposition for any tree~$T$. 
The \defn{treewidth} of a graph~$G$, denoted by~\defn{$\tw(G)$}, is the minimum width of a tree-decomposition of~$G$. Intuitively, the treewidth of $G$ measures how `tree-like' $G$ is. For example, trees have treewidth $1$, and outerplanar graphs have treewidth at most $2$. Treewidth is of fundamental importance in the graph minor theory of Robertson and Seymour and in algorithmic graph theory; see \citep{Bodlaender98,HW17,Reed03} for surveys on treewidth.

Denote by $P_n$ the path on $n$ vertices. Let $F_n$ be the \defn{fan} graph consisting of a path $P_n$ (called the \defn{base}) and one additional dominant vertex adjacent to every vertex along the path. If $v$ is the dominant vertex and $(w_1,\dots,w_n)$ is the base path, then the sequence $\{v,w_1,w_2\},\{v,w_2,w_3\},\dots,\{v,w_{n-1},w_n\}$ defines a $P_{n-1}$-decomposition of $F_n$ with width 2. This demonstrates that $\tw(F_n)\leq 2$. The following is another well-known fact about fans.

\begin{lem}\label{2ColouringFan}
Every 2-colouring of $F_n$ has clustering at least $\floor{\sqrt{n}}$.
\end{lem}
\begin{proof}
    Consider any red/blue colouring of $F_n$ with clustering $k$. Without loss of generality, the dominant vertex of $F_n$ is red. So there are at most $k-1$ red vertices in the base path~$P_n$ of~$F_n$. Thus there are at most $k$ blue components in $P_n$, each with at most $k$ vertices. Hence $n\leq (k-1)+k^2<(k+1)^2$, implying $k>\sqrt{n}-1$ and $k\geq\floor{\sqrt{n}}$.
\end{proof}

The \defn{cartesian product} of graphs $G$ and $H$, denoted by \defn{$G \CartProd H$}, is the graph with vertex set $V(G) \times V(H)$, where $(u,v)$ is adjacent to $(u',v')$ in $G \CartProd H$ if $u = u'$ and $vv'\in E(H)$, or $v = v'$ and $uu'\in E(G)$.

The \defn{strong product} of graphs $G$ and $H$, denoted by \defn{$G \boxtimes H$}, is the graph with vertex set $V(G) \times V(H)$, where $(u,v)$ is adjacent to $(u',v')$ in $G\boxtimes H$ if $u = u'$ and $vv'\in E(H)$, or $v = v'$ and $uu'\in E(G)$, or $uu'\in E(G)$ and $vv'\in E(H)$. 

Strong products and clustered colouring are inherently related. In particular, it follows from the definitions that a graph $G$ is $c$-colourable with clustering at most $k$ if and only if $G$ is contained in $H\boxtimes K_k$ for some $c$-colourable graph $H$. Here, a graph $G$ is \defn{contained} in a graph $G'$ if a subgraph of $G'$ is isomorphic to $G$. 

For an integer $m \geq 1$ and a graph $G$, define \defn{$\cone{m\,G}$} to be the graph constructed by taking $m$ pairwise disjoint copies of $G$ and adding one additional vertex adjacent to all vertices in each copy. Note that 
\begin{equation}\label{eq:twCone}
\tw( \cone{m\,G} ) \leq \tw(G)+1.    
\end{equation}
To see this, take a tree-decomposition of each copy of $G$, each with width $\tw(G)$. Add the new dominating vertex to each bag of each decomposition, and add some edges to connect the trees together, forming one large tree. This is a tree-decomposition of $\cone{m\, G}$ with width $\tw(G)+1$, hence $\tw(\cone{m\,G})\leq \tw(G)+1$.

Several proofs in this article use the following two results of~\citet{DvoWoo}. The first allows us to find separators in graphs with bounded treewidth.

\begin{lem}[\citep{DvoWoo}~Lemma 25]
\label{thm:DvoWooSeparators}
For any positive integers $n,t$ and any positive real number $p$, every graph $G$ on $n$ vertices with treewidth at most $t$ has a set $S$ of at most $p$ vertices such that each component of $G-S$ has at most $\frac{(t+1)n}{p}$ vertices. 
\end{lem}

The second lemma is a corollary of the isoperimetric inequality of \citet{BL91}. 

\begin{lem}[\citep{DvoWoo}~Lemma 13]
\label{CompGrid}
If $S$ is any set of at most $\frac{n^2}{2}$ vertices in $G:=P_n \CartProd P_n$ and $k\leq \frac{n^2}{e^2}$ and each component of $G-S$ has at most $k$ vertices, then $n^2\leq 4|S|k^{1/2}$.
\end{lem}

In addition, the constructions in this paper typically involve taking the strong products of graphs that contain long paths, producing grid structures. To analyse clustered colourings of these, we use the following result known as the Hex Lemma; see~\citep{Gale79,HT19} for background, and see~\citep[Proposition 6.1.4]{MatNes98} for the precise version of the statement below.

\begin{lem}[Hex Lemma]\label{lemma:HEX} 
Let $G$ be a plane internal triangulation with outer-cycle $(a,b,c,d)$. In every $2$-colouring $\chi$ of $G$ with $\chi(a)=\chi(c)$ and $\chi(b)=\chi(d)$, there is a monochromatic path either between $a$ and $c$ or between $b$ and $d$.
\end{lem}


\section{Two colours}\label{sec:TwoColours}

We start by studying $2$-colourings of graph products, proving~\cref{thm:TwoColoursUpperAndLower}. In this case, we demonstrate that the upper bound with no maximum degree restriction and the lower bound with a bounded degree graph are asymptotically tight up to a multiplicative constant. \cref{thm:TwoColoursUpperAndLower} follows immediately from~\cref{thm:TwoColoursUpper,thm:TwoColoursLower}.

\begin{lem}\label{thm:TwoColoursUpper}
For any graphs $H_1$ and $H_2$ both with treewidth at most $t$, the graph $G:= H_1\boxtimes H_2$ is $2$-colourable with clustering at most $2((t+1)|V(G)|)^{2/3}$. 
\end{lem}

\begin{proof}
Let $n_i:=|V(H_i)|$ for $i\in\{1,2\}$. Let 
$n:=|V(G)|=n_1n_2$. By \cref{thm:DvoWooSeparators}, for each $i\in\{1,2\}$, there is a set $S_i$ of at most $(t+1)^{2/3}n_i/n^{1/3}$ vertices in $H_i$ such that each component of $H_i-S_i$ has at most $((t+1)n)^{1/3}$ vertices. Let $X:= (S_1\times V(H_2)) \cup (V(H_1)\times S_2)$. So $$|X|\leq |S_1|n_2 + |S_2|n_1 \leq
(t+1)^{2/3}n_1n_2/n^{1/3} + (t+1)^{2/3}n_2n_1/n^{1/3} = 
2((t+1)n)^{2/3}.$$ Colour each vertex in $X$ blue and colour all other vertices red. Each red component $Y$ is contained in $Y_1\boxtimes Y_2$ for some component $Y_1$ of $H_1-S_1$ and some component $Y_2$ of $H_2-S_2$. Thus $|V(Y)|\leq |V(Y_1\boxtimes Y_2)| \leq ((t+1)n)^{2/3}$. 
Hence $G$ is 2-coloured with clustering
at most $2((t+1)n)^{2/3}$. 
\end{proof}

We now prove that~\cref{thm:TwoColoursUpper} is asymptotically tight, even when $H_2$ is a path.

\begin{lem}\label{thm:TwoColoursLower}
    There are infinitely many graphs $H$ with treewidth $2$ and paths $P$ such that every $2$-colouring of $H\boxtimes P$ has clustering at least $\frac{1}{2}|V(H\boxtimes P)|^{2/3}$.
\end{lem}

\begin{proof}
Fix any integer $n\geq 2$ and let $P:= P_n$, $H:= F_{n^2}$, and $G:=H\boxtimes P$ so that $|V(G)| = n^3$. With $x$ being the dominating vertex of the fan $H$, we refer to $\{(x, v) \colon v\in P\}$, the set of copies of $x$ in $G$, as the \defn{spine} of $G$. 

Suppose, toward a contradiction, that there exists a $2$-colouring of $G$ in red and blue with clustering less than $\frac{1}{2}n^2$. Suppose that two adjacent vertices of the spine are coloured differently. Then there is a red vertex and a blue vertex of $G$ both dominating the same copy of $P_{n^2}$. Without loss of generality, at least $\frac{1}{2}n^2$ of these vertices are coloured red, and they are all connected into one component by the dominating red vertex, contradicting the fact that the colouring has clustering less than $\frac{1}{2}n^2$.

It follows that the entire spine is monochromatic. Let this colour be red. Since vertices along the spine are dominating in their respective copies of $H$, this means that all red vertices in the graph $G$ are joined into a single component and therefore there are less than $\half n^{2}$ red vertices in $G$.

Consider the crossed $(n^2\times n)$-grid, $P_{n^2} \boxtimes P_n$, where the first factor is the base of the fan $H$. At most $\half n^2$ columns contain a red vertex. So at least $\frac{1}{2}n^{2}$ columns are all-blue. Moreover, there is a row $R$ with less than $\half n$ red vertices. In the subgraph induced by $R$ there are at most $\half n+1$ blue components. At least one of these components intersects at least $(\half n^2)/ ( \half n +1)  \geq \half n$ all-blue columns, which are thus contained in a blue component with at least $(\half n) n=\half n^2$ vertices. 
\end{proof}

Note that the bound $\tw(H)\leq 2$ in~\cref{thm:TwoColoursLower} is best possible. In particular, for any trees $T_1$ and $T_2$, by~\cref{thm:project} there exists a $2$-colouring of $T_1\boxtimes T_2$ with clustering at most $|V(T_1\boxtimes T_2)|^{1/2}$.

We present an alternative proof that the upper bound in~\cref{thm:TwoColoursUpper} is asymptotically tight, using a symmetric construction.

\begin{lem}\label{thm:SymmetricLowerBoundTwoColours}
For any integer $n\geq 4$, every $2$-colouring of $F_n\boxtimes F_n$ has clustering at least~$\frac13 n^{4/3}$.
\end{lem}

\begin{proof}
Let $A$ and $B$ be copies of $F_n$, with $V(A)=\{a_0,a_1,\dots,a_n\}$ and $V(B)=\{b_0,b_1,\dots,b_n\}$, where $a_0$ dominates the path $a_1,\dots,a_n$ in $A$, and $b_0$ dominates the path $b_1,\dots,b_n$ in $B$. The graph $A\boxtimes B$ contains an induced copy of $P_n\boxtimes P_n$ on the vertices $(a_i,b_j)$ with $1\leq i,j\leq n$, and consider its subgraph $H\cong P_n\CartProd P_n$.

Suppose, for the sake of contradiction, that there exists a 2-colouring of $A\boxtimes B$ with clustering less than $k:=\frac13 n^{4/3}$. Without loss of generality, assume that $(a_0,b_0)$ is blue. Note that, since $n\geq 4$, we have $k\leq\frac{n^2}{2}$ and $k\leq \frac{n^2}{e^2}$. 

Let $S$ be the set of blue vertices in $H$. Since $(a_0,b_0)$ dominates $A\boxtimes B$, we have $|S|\leq k$. 
Each component of $H-S$ is contained in a monochromatic component of $A\boxtimes B$, and thus has at most $k$ vertices. By \cref{CompGrid},
\[n^2\leq 4|S|k^{1/2} \leq 4 k^{3/2} \leq 4 (\tfrac13 n^{4/3})^{3/2} < n^2,\]
which is the desired contradiction. 
\end{proof}


\section{Three colours}\label{sec:ThreeColours}

We now move to $3$-colourings of graph products,
proving~\cref{thm:ThreeColoursMaxDegreeUpperAndLower,thm:ThreeColoursUpperAndLower}. 

\subsection{Unrestricted maximum degree}

We start by proving~\cref{thm:ThreeColoursUpperAndLower}, where $H_1,H_2$ both have bounded treewidth. For the upper bound, we build on the $2$-colour construction from \cref{thm:TwoColoursUpper} to prove the following result with three colours. This approach leads to a bound for all $c$ by induction. However, for $c\geq 4$,~\cref{thm:UpperNoMaxDegree} yields a better bound than would be achieved with this strategy.

\begin{lem}
For any graphs $H_1$ and $H_2$ both with treewidth at most $t$, 
the graph $G:= H_1\boxtimes H_2$ is $3$-colourable with clustering at most $2(t+1)^{6/7}|V(G)|^{4/7}$. 
\end{lem}

\begin{proof}
Let $n_i:=|V(H_i)|$ for $i\in\{1,2\}$. Let $n:=|V(G)|=n_1n_2$.
By \cref{thm:DvoWooSeparators}, for each~$i\in\{1,2\}$, there is a set $S_i$ of at most $(t+1)^{6/7}n_i/n^{3/7}$ vertices in $H_i$ such that each component of $H_i-S_i$ has at most $(t+1)^{1/7}n^{3/7}$ vertices. Let $X:= (S_1\times V(H_2)) \cup (V(H_1)\times S_2)$.~So 
 $$|X|\leq |S_1|n_2 + |S_2|n_1 \leq
(t+1)^{6/7}n_1n_2/n^{3/7} 
+ (t+1)^{6/7}n_2n_1/n^{3/7} = 
2 (t+1)^{6/7}n^{4/7}.$$ Each component $Y$ of $G-X$ is contained in $Y_1\boxtimes Y_2$ for some component $Y_1$ of $H_1-S_1$ and some component $Y_2$ of $H_2-S_2$, implying that $|V(Y_1\boxtimes Y_2)|\leq (t+1)^{2/7}n^{6/7}$. By~\cref{thm:TwoColoursUpper}, $Y_1\boxtimes Y_2$ (and hence also $Y$) is $2$-colourable with clustering at most $$
2((t+1)|V(Y_1\boxtimes Y_2)|)^{2/3} \leq
2((t+1) (t+1)^{2/7}n^{6/7})^{2/3} \leq
2(t+1)^{6/7}n^{4/7}.$$ By using a third colour for $X$, we obtain a 3-colouring of $G$ with clustering
at most $2(t+1)^{6/7}n^{4/7}$.
\end{proof}

The next lemma proves the lower bound in~\cref{thm:ThreeColoursUpperAndLower}, and shows that the upper bound in~\cref{thm:ThreeColoursMaxDegreeUpperAndLower} for the bounded degree scenario is not optimal in the unrestricted scenario.

\begin{lem}
\label{FanFan3Colouring}
    Every $3$-colouring of $F_n\boxtimes F_n$ has clustering at least $\parens*{1-\frac{1}{\sqrt{2}}} n$.
\end{lem}

\begin{proof}
Let $A$ and $B$ be copies of $F_n$, with $V(A)=\{a_0,a_1,\dots,a_n\}$ and $V(B)=\{b_0,b_1,\dots,b_n\}$, where $a_0$ dominates the path $a_1,\dots,a_n$ in $A$, and $b_0$ dominates the path $b_1,\dots,b_n$ in $B$. Suppose there exists a red/blue/green-colouring of $G:=A\boxtimes B$ with clustering less than $\delta n$, where $\delta:=1-\frac{1}{\sqrt{2}}$. We may assume without loss of generality that the dominating vertex $(a_0,b_0)$ is red. In particular, there are less than $\delta n$ red vertices in $G$.

Let $X=\{(a_0,b_i)\colon i\in[n]\}$ and $Y=\{(a_i,b_0)\colon i\in[n]\}$. Note that for any $i,j\in[n]$, the vertices $(a_0,b_i)$ and $(a_j,b_0)$ are adjacent in $G$. Assume that $X$ contains at least one blue vertex and at least one green vertex. Since each vertex of $X$ dominates $Y$, it follows that $Y$ contains less than $\delta n$ blue and less than $\delta n$ green vertices. This is a contradiction, since there are less than $\delta n$ red vertices in $G$ and $\delta<1/3$ while $|Y|=n$. Therefore, without loss of generality, $X$ contains more than $(1-\delta)n$ blue vertices and no green vertices.

We can apply similar reasoning to $Y$. Given that $X$ contains at least one blue vertex, $Y$ has less than $\delta n$ blue vertices, and therefore $Y$ contains more than $(1-\delta)n$ green vertices and no blue vertices.

Now focus on the subgraph $H$ of $G$ induced by $\{(a_i,b_j):i,j\in[n]\}$, which forms a copy of $P_n\boxtimes P_n$. Each vertex of $X$ dominates a row of $H$, and each vertex of $Y$ dominates a column. Since there are less than $\delta n$ red vertices in $G$, there exist more than $(1-\delta)n$ rows with no red vertices, each dominated by a blue vertex from $X$. Each such row contains less than $\delta n$ blue vertices, and hence more than $(1-\delta)n$ green vertices. Therefore, $H$ contains more than $(1-\delta)^2n^2$ green vertices. If we consider columns dominated by green vertices, an identical argument yields that there are more than $(1-\delta)^2n^2$ blue vertices in $H$. Since $2(1-\delta)^2=1$, there are more than $n^2$ vertices in $H$ coloured blue or green, but this is a contradiction since $H$ has only $n^2$ vertices.
\end{proof}

Note that the clustering value in~\cref{FanFan3Colouring} is best possible (up to the multiplicative constant) for $3$-colourings of $F_n\boxtimes F_n$. To see this, colour all the high-degree vertices red. What remains is $P_n\boxtimes P_n$, which can be 2-coloured by alternating rows. So $F_n\boxtimes F_n$ is $3$-colourable with clustering $n$.

\subsection{Bounded maximum degree}\label{sec:ThreeColoursLowerMaxDegree}

We now focus on~\cref{thm:ThreeColoursMaxDegreeUpperAndLower}, when $H_2$ has bounded treewidth and bounded maximum degree. The upper bound follows directly from the general $c\geq 2$ result,~\cref{thm:UpperBoundWithMaxDegree}, presented in~\cref{sec:ArbitraryNbColours}. We prove that this upper bound is asymptotically tight (up to the multiplicative factor) when $c=3$, which concludes the proof of~\cref{thm:ThreeColoursMaxDegreeUpperAndLower}.

Recall that $F_n$ is the fan graph with base path $P_n$. Define \defn{$H_n$} $:=\cone{n^2\, F_{n^4}}$, and let $x$ be the dominating vertex in $H_n$ (see~\cref{fig:cone}). The \defn{fans of $H_n$} refer to the $n^2$ pairwise disjoint copies of $F_{n^4}$ in $H_n$. Note that $H_n$ is planar, and $\tw(H_n)\leq \tw(F_{n^4})+1=3$ by~\cref{eq:twCone}.

\begin{figure}[ht]
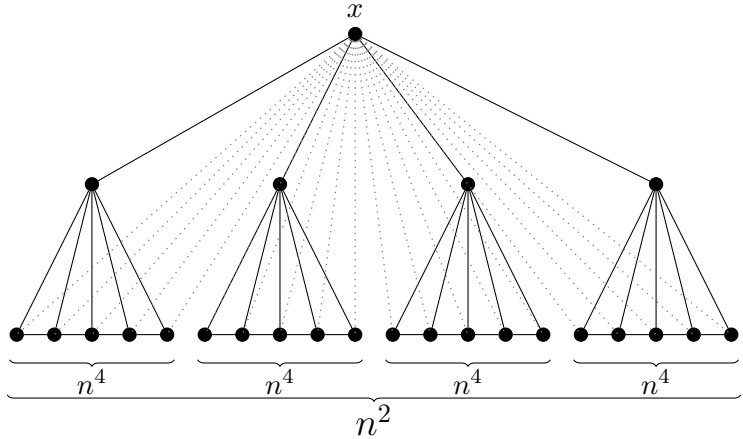

    \centering
    \FanBrace{0.5}
    \caption{ $H_n = \widehat{n^2\, F_{n^4}}$ }
    \label{fig:cone}
\end{figure}

\begin{lem}\label{thm:3lowerbound}
For any $\delta<1/2$ and any $n$ large enough, every $3$-colouring of $G:=H_n\boxtimes P_n$ has clustering at least $\delta n^3$.
\end{lem}

Since $|V(G)|=n^7$ and $\tw(H_n)\leq 3$,~\cref{thm:3lowerbound} proves the lower bound of~\cref{thm:ThreeColoursMaxDegreeUpperAndLower}.

\textbf{Proof overview:} The proof is by contradiction. We first show in Claim~\ref{claim:monospine} that a specific set of vertices of $G$ must be monochromatic, say red. This allows us to bound the total number of red vertices in $G$. We then restrict our attention to a copy $J$ of $F_{n^4}\boxtimes P_n$ with fewer than $\delta n$ red vertices. We recognise two subgraphs in $J$: a \emph{path} induced by the copies of the dominant vertex of $F_{n^4}$ and a \emph{grid} $L$ formed by the strong product of the path of $F_{n^4}$ with $P_n$. As there are few red vertices, linearly many vertices of the path of $J$ must be of the same non-red colour, say blue. In Claim~\ref{claim:domcomp} we argue moreover that these blue vertices all lie in the same blue component $S$ of $J$. We then consider the intersection $K$ of rows with blue associated dominant vertex and columns disjoint from $S$ in $L$, and note in Claim~\ref{claim:removeblue} that vertices $u,v\in K$ in the same component of $L-S$ remain so when deleting all blue vertices. This allows us to conclude the proof as, loosely speaking, the red vertices split $K$ into few components so there must be a large monochromatic component in the last remaining colour.

\begin{proof}
Let $H:= H_n$. We refer to $\{(x, v) \colon v\in P_n\}$, the set of copies of $x$ in $G$, as the \defn{spine} of $G$. Suppose, toward a contradiction, that the vertices of $G$ are $3$-coloured in red, blue and green, with clustering at most $\delta n^3$.


\begin{claim}\label{claim:monospine}
The spine of $G$ is monochromatic.
\end{claim}

\begin{proof}
Suppose that the spine of $G$ is not monochromatic. Let $r,b$ be two adjacent vertices of $P_n$ such that $(x,r)$ and $(x,b)$ have different colours, say red and blue respectively. In $G$, every vertex $w\in P_n$ has an associated copy $H^{(w)}$ of $H$ on vertex set $\{(z,w) \colon z\in V(H)\}$ (and edges between two vertices $(y,w)$ and $(z,w)$ precisely when $y$ and $z$ are adjacent in~$H$). Since $r$ and $b$ are adjacent in $P_n$, then $(x,r)$ and $(x,b)$ are both dominating vertices for~$H^{(r)}$. Therefore, as $G$ contains no monochromatic component of size $\delta n^3$, there must be less than~$\delta n^3$ red and less than $\delta n^3$ blue vertices in $H^{(r)}$.

Let $\mathcal{F}^{(r)}$ be the set of fans of $H^{(r)}$.
As $|\mathcal{F}^{(r)}|=n^2$, by averaging there is some $F\in \mathcal{F}^{(r)}$ which contains less than $2\delta n<n$ vertices coloured red or blue. Hence, the path of length~$n^4$ in $F$ contains a subpath of green vertices of length at least $n^3$, a contradiction.
\end{proof}


Without loss of generality, we may assume that all vertices of the spine are coloured red. Since the spine is a dominating set for $G$, there are less than $\delta n^3$ red vertices in $G$.
Let $\mathcal{F}$ be the set of fans in $H$. By averaging, there is an $F\in \mathcal{F}$ such that the subgraph $J$ of $G$ induced by $V(F)\times V(P_n)$ contains less than $\delta n$ red vertices. 

Consider the restriction of the colouring of $G$ to $J$. Note that $J$ contains a copy of the grid $P_{n^4} \boxtimes P_n$, formed by the strong product of the path of $F$ with $P_n$ (see the graph $G[J]$ on the left of Figure~\ref{fig:subgraphAndSubgrid}).
In this grid, rows are copies of the path of $F$ and columns are of the form $\{v\}\boxtimes P_n$, where $v$ is a vertex of the path of $F$.

Let $y$ be the dominating vertex of the fan $F$ and call a vertex of $J$ of the form $(y,v)$ with $v\in V(P_n)$ a \defn{top vertex}. 

As previously, for $v\in V(P_n)$, denote by $F^{(v)}$ the copy of $F$ associated to $v$ in $J$.
Since there are less than $\delta n$ red vertices in $J$, we may assume without loss of generality that there are at least $\frac{1-\delta}{2}n$ blue top vertices.

\begin{claim}\label{claim:domcomp}
All blue top vertices of $J$ are in the same monochromatic blue component in~$J$.
\end{claim}

\begin{proof} 
Let $(y,s)$, $(y,t)$ be two blue top vertices in $J$, for some $s,t\in [n]$. We show that there exists a blue path between $(y,s)$ and $(y,t)$ in $J$. Let $R_s$ and $R_t$ be the vertices in the paths of the fans $F^{(s)}$ and $F^{(t)}$ respectively, both of size $n^4$. Note that in $J$, these two rows bound a subgrid of $L$. Let $C$ denote the first column of this subgrid, and $C'$ the last.

\begin{figure}[ht]
    \centering
    \begin{subfigure}{0.60\textwidth}
        \includegraphics[scale=0.6]{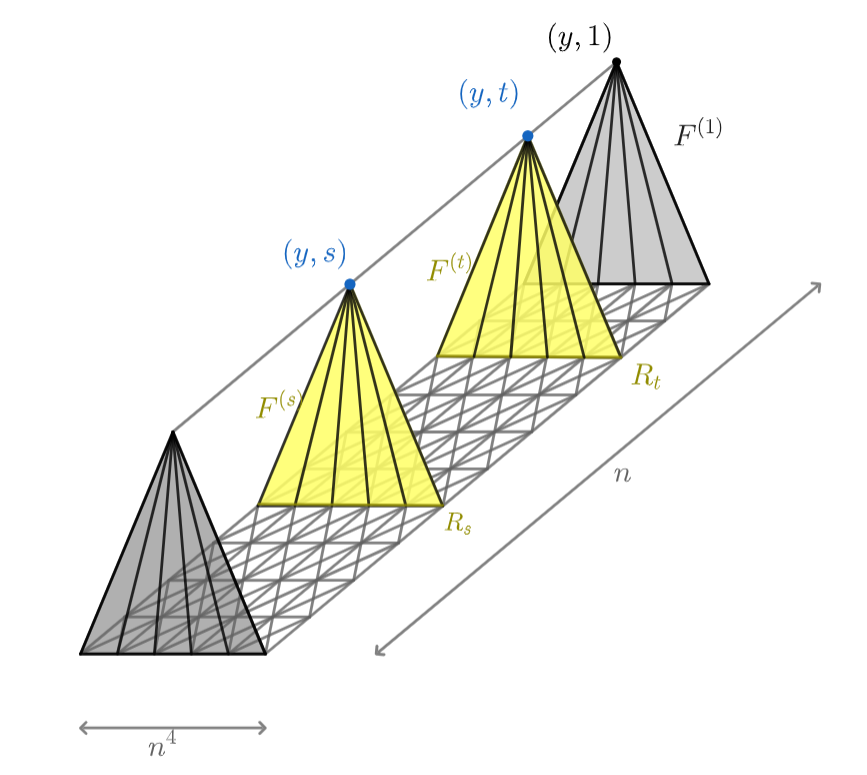}
    \end{subfigure}
    \begin{subfigure}{0.35\textwidth}
        \triggrid{5}{0.65}{0}
    \end{subfigure}
    \caption{Subgraph $G[J]$ and the subgrid of $L$ bounded by $R_s$ and $R_t$.}
    \label{fig:subgraphAndSubgrid}
\end{figure}

Add two dummy vertices $u_s,u_t$ adjacent to all vertices in $R_s$ and $R_t$ (respectively). Colour them blue. Add two dummy vertices $v,v'$ adjacent to all vertices in $C$ and $C'$ (respectively). Colour them red or green (see Figure~\ref{fig:SubgridWithDummy} for an illustration).

\begin{figure}[ht]
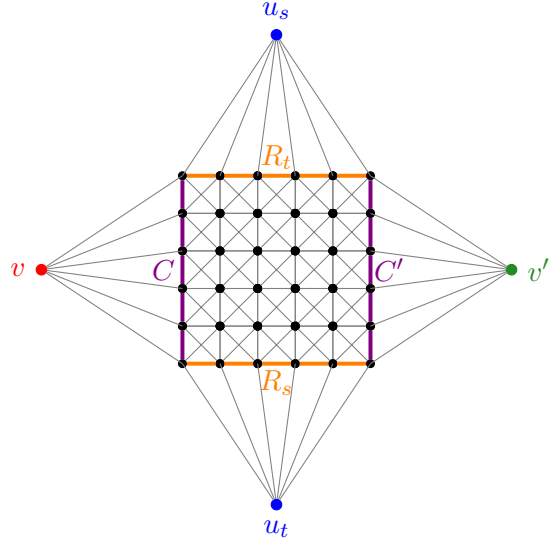

    \centering
    \triggrid{5}{0.5}{1}
    \caption{Subgrid of $L$ bounded by $R_s$ and $R_t$ with dummy vertices.}
    \label{fig:SubgridWithDummy}
\end{figure}

Then by the Hex Lemma (\cref{lemma:HEX}), either there is a blue path from $u_s$ to $u_t$, or there is a path from $v$ to $v'$ with no blue vertices. Suppose the latter occurs. As previously, this yields a path of length $n^4$ with no blue vertices. Since there are less than $\delta n$ red vertices in~$J$, it contains a subpath of $n^3$ green vertices and we are done.
We may therefore assume that there is a blue path from $u_s$ to $u_t$. This blue path intersects both $R_s$ and $R_t$. Since $(y,s)$ (respectively $(y,t)$) is dominating for $R_s$ (respectively $R_t$), this extends to a blue path from $(y,s)$ to $(y,t)$, as claimed.
\end{proof}


Let $S$ be the vertices of the monochromatic blue component of $J$ containing all blue top vertices. Note that there may be blue vertices in $J$ which are not in $S$, and as $G$ contains no monochromatic component of size $\delta n^3$, we have that $|S| < \delta n^3$.

Say that a row of $L$ is \defn{blue dominated} if its associated top vertex is blue, and that a column of $L$ is \defn{$S$-free} if it does not intersect $S$. In particular, as there are $n^4$ columns in $L$ and less than $\delta n^3$ elements in $S$, for any $\varepsilon>0$, for $n$ large enough, there are $(1-\varepsilon)n^4$ columns which are $S$-free. Define the \defn{boundary} of $S$ in $J$, denoted by $\partial S$, to be the set of vertices outside of $S$ that are adjacent to a vertex in $S$. 


\begin{claim}\label{claim:removeblue}
Let $u,v\in \partial S$ be vertices simultaneously on blue dominated rows and $S$-free columns.
Suppose $u,v$ are in the same component $Q$ of the grid $L$ after the vertices of $S$ have been removed, and that $Q$ contains an $S$-free column. 
Then, there is a path consisting of red and green vertices between $u$ and $v$ in $L$.
In other words, $u,v$ remain in the same component of $L$ after removing all blue vertices.
\end{claim}
\begin{proof}

Let $Q$ be a component of $L-S$ that contains an $S$-free column.
Let $R_1$, $R_2$, $C_1$, $C_2$ be the top row, bottom row, leftmost column, and rightmost column of $L$, respectively.

Let $D$ be the smallest component of $L$ containing $Q$, bounded up and down by $R_1$ and $R_2$, left by a blue path from $R_1$ to $R_2$ (or $C_1$ if such a path does not exist), and right by a blue path from $R_1$ to $R_2$ (or $C_1$ if such a path does not exist).

We now restrict our attention to $D$. Let $R$ and $R'$ denote the vertices of its top and bottom rows, and fix a full $S$-free column $K$ of $Q$. By the Hex Lemma (Lemma~\ref{lemma:HEX}), either there is a blue path from $S_1$ to $S_2$ or a red/green path from $R$ to $R'$. 

Suppose there is a blue path $B$ from $S_1$ to $S_2$. If $S_1 = C_1$ and $S_2 = C_2$, then this path has length at least $n^4$, and is therefore a monochromatic component of size at least $\delta n^3$, a contradiction. Otherwise, at least one of $S_1$, $S_2$ is a blue monochromatic path crossing every row in $D$. Without loss of generality, suppose $S_1$ is such a path. As some rows of $L$ are blue dominated, this path is entirely in $S$. Hence, as $B$ is blue and touches $S_1$, $B\subset S$ as well. But $B$ certainly intersects $K$, an $S$-free column, which is impossible. 

Hence, there must be a path $P$ from $R$ to $R'$ whose vertices are not blue. 

It now suffices to show that for each $u\in Q$ which is the intersection of an $S$-free column with a blue dominated row, there is a path from $u$ to $P$ whose vertices are not blue.
Indeed, any $u,v\in Q$ are then connected with non-blue vertices by concatenating the paths from $u$ to $P$, from $v$ to $P$, and $P$ itself. 

If $u\in P$ we are done, so we may assume that $u\not\in P$.
Let $C^*$ and $R^*$ be the column and row of $u$, respectively.
We let
\begin{itemize}[noitemsep]
\item $P_1$ be the shortest path from $u$ to $P$ in $D$ that uses only vertices of $C^*$ above $u$ and possibly vertices of $R$, and $x\in P$ the other endpoint of $P_1$, and
\item $P_2$ be the shortest path from $u$ to $P$ in $D$ that uses only vertices of $C^*$ below $u$ and possibly vertices of $R'$, and $y\in P$ the other endpoint of $P_2$.
\end{itemize}

Let $A$ be the subgraph of $D$ induced by the vertices in the region enclosed by $P_1$, $P_2$ and the subpath of $P$ between $x$ and $y$. By the Hex Lemma (Lemma \ref{lemma:HEX}), either there is a blue path from $P_1 -u$ to $P_2-u$, or there is a red/green path from $u$ to $P$.

Suppose there is a blue path $P'$ from $P_1-u$ to $P_2-u$. Note that, as all vertices of $P_1$ lie above or on row $R^*$ and all vertices of $P_2$ lie below or on row $R^*$, any path from $P_1$ to $P_2$ must cross row $R^*$. In particular, $P'$ must be entirely in $S$ as it is blue. Consider the endpoints of $P'$.
If they lie in $R$ and $R'$, then, as $P'$ is blue and in $S$, both of its endpoints lie to one side of the $S$-free column $K\subseteq Q$. Thus, $P'$ splits $D$ into two non-trivial parts with one containing $Q$, contradicting the minimality of $D$.
Hence, by construction, some endpoint of $P'$ must lie on $C^*$. But $C^*$ is an $S$-free column, so this is also a contradiction. Therefore, there is a red/green path from $u$ to $P$, as desired.
\end{proof}

Let $K\subset \partial S$ be the set of vertices of $L$ simultaneously on blue dominated rows and $S$-free columns. Recall that $L$ contains at least $(1-\varepsilon)n^4$ $S$-free columns and at least $\frac{1-\delta}{2}n$ blue top vertices, therefore $\size{K}\geq \frac{(1-\varepsilon)(1-\delta)}{2}n^5$.

Since $\vert S\vert < \delta n^3$, at least one of the $n$ rows of the grid $L$ contains less than $\delta n^2$ vertices from $S$. It follows that the $S$-free columns of $L$ live in at most $\delta n^2$ components of $L-S$.

As $K$ contains only vertices from $S$-free columns, the vertices of $K$ live in at most $\delta n^2$ components of $L-S$. By Claim~\ref{claim:removeblue} the vertices of $K$ still live in at most $\delta n^2$ components of $L$ after all blue vertices are removed. We now remove the less than $\delta n$ red vertices. Note that for any $\varepsilon'>0$, for $n$ large enough, we have at least $(1-\varepsilon')\size{K}$ green vertices in $K$. As each vertex of $L$ has degree at most $8$, at each removal we add at most $7$ components. After removing all red vertices, for any $\varepsilon''>0$, for $n$ large enough, we still have less than $(\delta +\varepsilon'')n^2$ components. Therefore, using $\delta<1/2$ and choosing $\varepsilon$, $\varepsilon'$, and $\varepsilon''$ small enough, there is a component of size at least
\[ \frac{(1-\varepsilon)(1-\varepsilon')(1-\delta)}{2(\delta+\varepsilon'')}n^3 > \delta n^3\]
containing only green vertices, a contradiction.\qedhere

\end{proof}


\section{Arbitrary number of colours}
\label{sec:ArbitraryNbColours}

We now present our most general results, with an arbitrary number of colours. 

\subsection{Bounded maximum degree}
\label{sec:ArbitraryNbColoursMaxDegree}

\subsubsection{The upper bound}

In this section, we prove~\cref{thm:UpperBoundWithMaxDegree}, a general upper bound for products of bounded treewidth graphs, when one graph is additionally assumed to have bounded maximum degree. Specifically, we show the following. 

\begin{thm}\label{thm:GeneralUpperBoundMaxDegreePrecise}
For any positive integers $t,c,\Delta$, for any graph $H_1$ with treewidth at most~$t$, and any graph $H_2$ with treewidth at most $t$ and maximum degree at most $\Delta$, the graph $G:= H_1\boxtimes H_2$ is $c$-colourable with clustering at most 
$6^{c+1}(t+1)^2 \Delta^c\, |V(G)|^{c/(c^2-c+1)}$. 
\end{thm}

The following theorem is a slightly more precise version of a result by~\citet[Theorem~1.2]{LMST08}.

\begin{lem}
\label{cColourTW}
For any positive integers $c,t\geq 1$, every graph $G$ with treewidth at most $t$ has a $c$-colouring with clustering at most $(t+1)^{(c-1)/c}\,|V(G)|^{1/c}$. 
\end{lem}

As remarked by \citet[Theorem~1.2]{LMST08}, this bound is asymptotically optimal for fixed $c$, since there exists a graph $G_c$ with treewidth at most $c$, on $O(n^{c/2})$ vertices and such that every $c$-colouring has clustering $\Omega(\sqrt{n})$.

\begin{proof}[Proof of~\cref{cColourTW}]
We proceed by induction on $c$. The $c=1$ case is trivial. Now assume that $c\geq 2$.  Let $G$ be a graph with $n$ vertices and treewidth at most $t$. Let $p:=(t+1)^{1/c}n^{(c-1)/c}$. If $p\leq t+1$, then any colouring of $G$ has clustering at most $n\leq (t+1)^{(c-1)/c}n^{1/c}$ as desired. Now assume that $p\geq t+1$. By \cref{thm:DvoWooSeparators}, there exists a set $S$ of at most $p$ vertices in $G$ such that each component of $G-S$ has at most 
\[\frac{(t+1)n}{p} = 
\frac{(t+1)n }{ (t+1)^{1/c}n^{(c-1)/c} } =
(t+1)^{(c-1)/c}n^{1/c} \]
vertices. By induction, there exists a $(c-1)$-colouring of $G[S]$ with clustering
$$
(t+1)^{(c-2)/(c-1)} p^{1/(c-1)} = 
(t+1)^{(c-2)/(c-1)} \left((t+1)^{1/c}n^{(c-1)/c}\right)^{1/(c-1)} 
\leq (t+1)^{(c-1)/c} n^{1/c}.$$ 

Extend this $(c-1)$-colouring of $G[S]$ to a $c$-colouring of $G$ by giving one new colour to all vertices in $G-S$. This has clustering at most~$(t+1)^{(c-1)/c} n^{1/c}$.
\end{proof}

\cref{cColourTW} is only interesting if $c\leq t$, since if $c\geq t+1$, then $G$ is $(c-1)$-degenerate and thus properly $c$-colourable. 

We use the following three lemmas from the literature.

\begin{lem}[{\citep[Lemma~14]{EW23}}]\label{AddOneClique}
    If a graph $G$ is $c$-colourable with clustering $k$, then $G\boxtimes K_\ell$ is $c$-colourable with clustering $k\ell$.
\end{lem}

\begin{lem}[{\citep[Lemma~32]{EW23}}]\label{AddOneColour}
    If a graph $G$ is $c$-colourable with clustering $k$ and $T$ is a tree with maximum degree $\Delta$, then $G\boxtimes T$ is $(c+1)$-colourable with clustering less than $2k(\Delta - 1)^{c-1}$ if $\Delta\geq 3$, and clustering at most $ck$ otherwise.
\end{lem}

Note that \citet[Lemma~32]{EW23} mention only the case $\Delta\geq 3$, but the result when $\Delta<3$ is trivially extracted from their proof, since they showed that $G\boxtimes T$ is $(c+1)$-colourable with clustering at most $k(1+(\Delta-1)+(\Delta-1)^2+\dots+(\Delta-1)^{c-1})$.

\citet{DW22a} proved the following extension of a classical result of \citet{DO95}. 

\begin{lem}[{\citep[Theorem~2]{DW22a}}]
\label{lem:containedtree}
Every graph $H$ with treewidth at most $t$ and maximum degree at most $\Delta$ is contained in $T\boxtimes K_{18(t+1)\Delta}$ for some tree $T$ with maximum degree at most $6\Delta$ and with $|V(T)|\leq \max\{ |V(H)|/2(t+1),1\}$.
\end{lem}

We are now ready to prove~\cref{thm:GeneralUpperBoundMaxDegreePrecise}, which is a more precise version of \cref{thm:UpperBoundWithMaxDegree}. 

\begin{proof}[Proof of~\cref{thm:GeneralUpperBoundMaxDegreePrecise}] \Cref{cColourTW} implies that this statement is trivially true if $\Delta=0$. Assume now that $\Delta\geq 1$. By~\cref{lem:containedtree}, there exists a tree $T$ with $\Delta(T)\leq 6\Delta$ and $|V(T)|\leq \max\{ |V(H_2)|/2(t+1),1\}$ such that $G$ is contained in $H_1\boxtimes T \boxtimes K_{18(t+1)\Delta}$.

Define $F:= H_1\boxtimes T$ and let $h:=|V(H_1)|$, $\ell:=|V(T)|$, $f:=|V(F)|$ and $n:=|V(G)|$. Note that \(f = h\ell \leq \frac{n}{2(t+1)}.\) 

We start by showing that $F$ admits a $c$-colouring with a bound on its clustering. Assume first that $h\geq \ell^{c(c-1)}$. Then
\begin{align*}
    \ell^{c(c-1)^2} &\leq h^{c-1},\\
    h^{c^2-c+1} \ell^{c(c^2-c+1)} &\leq (h\ell)^{c^2},\\
    h^{1/c}\ell &\leq (h\ell)^{c/(c^2-c+1)}.
\end{align*}
By~\cref{cColourTW}, there exists a $c$-colouring of $H_1$ with clustering at most $(t+1)^{(c-1)/c} h^{1/c}$. For each $x\in V(H_1)$ and each $y\in V(T)$, colour $(x,y)\in V(F)$ by the colour assigned to $x$. We obtain a $c$-colouring of $F$ with clustering at most
\[(t+1)^{(c-1)/c}h^{1/c}\ell\leq (t+1)^{(c-1)/c} (h\ell)^{c/(c^2-c+1)}.\] 

Now assume that $h\leq \ell^{c(c-1)}$. Hence, 
\begin{align*}
    h^{(c^2-c+1)} &\leq (h\ell)^{c(c-1)},\\
    h^{1/(c-1)} &\leq (h\ell)^{c/(c^2-c+1)}.
\end{align*}
By~\cref{cColourTW}, $H_1$ admits a $(c-1)$-colouring with clustering $(t+1)^{(c-1)/c}h^{1/(c-1)}$. Then, with $\Delta\geq 1$, it follows from~\cref{AddOneColour} that $F$ has a $c$-colouring with clustering at most
\[2(t+1)^{(c-1)/c}h^{1/(c-1)}(6\Delta-1)^{c-1} \leq 
2(t+1)^{(c-1)/c}(6\Delta-1)^{c-1} (h\ell)^{c/(c^2-c+1)}.\]

Finally, let 
\[\alpha := 2(t+1)^{(c-1)/c}(6\Delta-1)^{c-1} 18 (t+1)\Delta\left(\frac{1}{2(t+1)}\right)^{c/(c^2-c+1)}
\leq6^{c+1}(t+1)^2 \Delta^c.\]
It follows from~\cref{AddOneClique} that $G\subseteq H_1\boxtimes T\boxtimes K_{18(t+1)\Delta}$ has a $c$-colouring with clustering at most $\alpha n^{c/(c^2-c+1)}$.\end{proof}

\subsubsection{A lower bound}

We now turn our attention to the lower bound for this general case, proving~\cref{thm:GeneralLowerBoundMaxDegree}.

\begin{lem}
\label{lem:inductionGeneralLowerBound}
For any positive integers $c,k\geq 2$, let $H$ be a graph and $P$ be a path such that every $(c-1)$-colouring of $H$ has clustering at least $k$, and every $c$-colouring of $H\boxtimes P$ has clustering at least $k$. Let $J:= \cone{(k-1)\, H}$. Then  every $(c+1)$-colouring of $J\boxtimes P$ has clustering at least $k$. 
\end{lem}

\begin{proof}
By definition,  $J$ is obtained from $k-1$ disjoint copies $H_1,\dots,H_{k-1}$ of $H$ by adding one dominant vertex.  
Say $P=(v_1,\dots,v_p)$. Let $J^i$ be the copy of $J$  in $J\boxtimes P$ corresponding to $v_i$. For $j\in\{1,\dots,k-1\}$, let $H^i_j$ be the copy of $H_j$ in $J^i$, and call $X_j:=V(H^1_j\cup\dots\cup H^p_j)$ a \defn{column}. Consider any $(c+1)$-colouring of $J\boxtimes P$. 

First, suppose there exist consecutive vertices $v_i$ and $v_{i+1}$ that are assigned distinct colours, say red and blue respectively. In $J\boxtimes P$, each of $v_i$ and $v_{i+1}$ dominate $H^i_1,\dots,H^i_{k-1},H^{i+1}_1,\dots,H^{i+1}_{k-1}$. If at least $k-1$ of $H^i_1,\dots,H^i_{k-1},H^{i+1}_1,\dots,H^{i+1}_{k-1}$ contain a red vertex, then with $v_i$ we have a red component on at least $k$ vertices, as desired. If at least $k-1$ of $H^i_1,\dots,H^i_{k-1},H^{i+1}_1,\dots,H^{i+1}_{k-1}$ contain a blue  vertex, then with $v_{i+1}$ we have a blue component on at least $k$ vertices, as desired. Now assume that at most $k-2$ of $H^i_1,\dots,H^i_{k-1},H^{i+1}_1,\dots,H^{i+1}_{k-1}$ contain a red  vertex, and at most $k-2$ of $H^i_1,\dots,H^i_{k-1},H^{i+1}_1,\dots,H^{i+1}_{k-1}$ contain a blue  vertex. Thus, at least one of 
$H^i_1,\dots,H^i_{k-1},H^{i+1}_1,\dots,H^{i+1}_{k-1}$ contains no red vertex and no blue vertex, and is therefore $(c-1)$-coloured. By assumption, this copy of $H$ has a monochromatic component on at least $k$ vertices, as desired.

Now assume that $v_1,\dots,v_p$ are monochromatic, say red. Since $v_1,\dots,v_p$ is a connected dominating set in $J\boxtimes P$, the red subgraph is connected. Thus, at most $k-1$ vertices are red. Hence, at most $k-2$ vertices not in $\{v_1,\dots,v_p\}$ are red. In particular, at most $k-2$ columns contain a red vertex, so some column $X_j$ contains no red vertex. The subgraph of $J\boxtimes P$ induced by $X_j$, which is isomorphic to $H\boxtimes P$, is $c$-coloured, and thus has clustering at least $k$, as desired. 
\end{proof}

The next lemma is folklore, see for instance~\citep{vdHW18,NSSW19,OOW19}.

\begin{lem}
\label{ConeLemma}
For any positive integers $c,k\geq 2$, if $H$ is a graph such that every $c$-colouring of $H$ has clustering at least $k$, then every $(c+1)$-colouring of $\cone{(k-1)\,H}$ has clustering at least $k$. 
\end{lem}

\begin{proof}
Consider any $(c+1)$-colouring of $\cone{(k-1)\,H}$. Say the dominant vertex is blue. If every copy of $H$ contains a blue vertex, then the blue component has at least $k$ vertices, as desired. Otherwise, some copy of $H$ is $(c-1)$-coloured, and thus has clustering at least $k$ by assumption.
\end{proof}

The following is a variant of a result by \citet{LMST08}.
\begin{lem}\label{HcUpdated}
For any positive integer $n,k \geq 2$ such that $k\leq n^3$, let $G_2=\cone{n^2\,F_{n^4}}$, and for $c\geq 2$, let $G_{c+1}:= \cone{(k-1)\,G_{c}}$. Then for every $c\geq 2$, every $c$-colouring of $G_c$ has clustering at least $k$. 
\end{lem}

\begin{proof}
We proceed by induction on $c\geq 2$. For the base case $c=2$, suppose there exists a red/blue colouring of $G_2$ with clustering less than $k$. Say the dominant vertex of $G_2$ is red. So there exists a copy of $F_{n^4}$ in $G_2$ with less than $k/n^2$ red vertices. The path of length $n^4$ in this copy of $F_{n^4}$ contains a subpath of length at least $n^6/k\geq n^3$, whose vertices are all blue, a contradiction. 
Now assume that $c\geq 2$ and every $c$-colouring of $G_c$ has clustering at least~$k$. By~\cref{ConeLemma}, every $(c+1)$-colouring of $G_{c+1}$ has clustering at least $k$, as desired. 
\end{proof}

We now prove the following precise version of~\cref{thm:GeneralLowerBoundMaxDegree}, which  follows immediately since~\cref{eq:twCone} implies that $\tw(G_{c})= \tw(G_{c-1})+1 = c+1$.

\begin{prop}\label{thm:GeneralLowerBoundMaxDegreePrecise}
For any large enough integer $n$, and any integers $k,c\geq 3$ such that $2k< n^3$ , let $G_2:=\cone{n^2\,F_{n^4}}$, $G_{c}:= \cone{(k-1)\,G_{c-1}}$, and let $\delta<1/2$ such that $k=\delta n^3$. Then every $c$-colouring of $G_{c-1}\boxtimes P_n$ has clustering at least 
\( \delta^{\frac{7}{3c-2}}|V(G_{c-1}\boxtimes P_n)|^{1/(c-\frac{2}{3})}.\)
\end{prop}

\begin{proof} 
First note that $|V(G_{c-1}\boxtimes P_n)|=\delta^{c-3}n^{3c-2}$ for any $c\geq 3$. Therefore
\[\delta^{\frac{7}{3c-2}}|V(G_{c-1}\boxtimes P_n)|^{1/(c-\frac{2}{3})}= 
\delta^{\frac{7}{3c-2}} \delta^{\frac{3(c-3)}{3c-2}} n^3 = \delta n^3 = k.\]

We reason by induction on $c$. The base case $c=3$ is true by~\cref{thm:3lowerbound}. Assume now that the statement is true for some $c\geq 3$, therefore that every $c$-colouring of $G_{c-1}\boxtimes P_n$ has clustering at least $k$. It follows from~\cref{HcUpdated} that every $(c-1)$-colouring of $G_{c-1}$ has clustering at least $k$. Therefore~\cref{lem:inductionGeneralLowerBound} implies that every $(c+1)$-colouring of $G_{c}\boxtimes P_n$ has clustering at least $k$.
\end{proof}

\subsection{Unrestricted maximum degree}

By colouring each graph separately with the smallest clustering possible, and then using a product colouring, we prove the following statement that implies the upper bound of~\cref{thm:GeneralNoMaxDegree}.

\begin{lem}\label{thm:UpperNoMaxDegree}
For any positive integers $c,t\geq 1$ and any graphs $H_1$ and $H_2$ both with treewidth at most $t$, the graph $H_1\boxtimes H_2$ is $c$-colourable with clustering at most
\[(t+1)^{2(1-1/\sqrt{c})}\,
|V(H_1\boxtimes H_2)|^{1/\floor{\sqrt{c}}}.\] 
\end{lem}

\begin{proof}
    Let $s=\floor{\sqrt{c}}$. Given $H_1$ and $H_2$ both with treewidth at most $t$, let $n_i:=|V(H_i)|$, $n=n_1n_2$, apply~\cref{cColourTW} to each $H_i$ with $s$ colours, and take the product colouring. So $H_1\boxtimes H_2$ is $s^2$-colourable with clustering at most 
\[(t+1)^{(s-1)/s}\, n_1^{1/s}(t+1)^{(s-1)/s}\, n_2^{1/s} \leq (t+1)^{2(\sqrt{c}-1)/\sqrt{c}}\, n^{1/\floor{\sqrt{c}}}.\qedhere\] 
\end{proof}

We now prove the lower bound of~\cref{thm:GeneralNoMaxDegree}. Fix a positive integer $t$. Consider graphs $H_1$ and $H_2$ with treewidth at most $t$. So both $H_1$ and $H_2$ are properly $(t+1)$-colourable. 
A product colouring shows that $H_1\boxtimes H_2$ is properly $(t+1)^2$-colourable. That is, $H_1\boxtimes H_2$ is $(t+1)^2$-colourable with clustering 1. \citet{EW23} showed that in such a result with bounded clustering, $(t+1)^2$ colours is best possible. That is, there exists a family of graphs $H_1,H_2$ with treewidth at most $t$ such that for any $c<(t+1)^2$, in any $c$-colouring of $H_1\boxtimes H_2$ the clustering must increase with $|V(H_1\boxtimes H_2)|$. We extract the following result from their proof.

\begin{lem}\label{LowerBoundProductEW}
    For every integer $c\geq 2$, there exists infinitely many graph $H_1,H_2$ with treewidth at most $\sqrt{c}$ such that 
    every $c$-colouring of $G:=H_1\boxtimes H_2$ has clustering $\Omega(|V(G)|^{1/2\sqrt{c}})$.
\end{lem}

\begin{proof}
For $t,n\in \NN$, define the graph $C_{t,n}$ recursively as follows. Let $C_{1,n}:=K_1$, and $C_{t,n}:=\cone{n\,C_{t-1,n}}$ for $t\geq 2$. 
Given $c$, let $t$ be the integer such that $t^2\leq c<(t+1)^2$. 
By \cref{eq:twCone}, $\tw(C_{t+1,n})\leq t\leq\sqrt{c}$. Let $G:=C_{t+1,n}\boxtimes C_{t+1,n}$, where $n\gg t$. 
Note that $|V(C_{t+1,n})|=n^{t} + O(n^{t-1})$ and thus
$|V(G)|=n^{2t} + O(n^{2t-1})$. 
\citet[Theorem~16]{EW23}\footnote{\citet{EW23} actually work in the more general setting of fractional $p:q$-colourings, where a $c$-colouring is equivalent to a $c:1$-colouring. They also define $C_{t,n}$ as the closure of a rooted tree, which is equivalent to our recursive definition.} showed that in any $c$-colouring of $G$,  
there is a monochromatic component with maximum degree at least $\frac{n}{(t+1)^412^{t+1}}$.  Thus, for fixed $t$, any $c$-colouring of $G$ has clustering     \[\Omega(n)=\Omega(|V(G)|^{1/2t})=\Omega(|V(G)|^{1/2\sqrt{c}}).\qedhere\]
\end{proof}

Consider graphs $H_1$ and $H_2$ with bounded treewidth on $n_1$ and $n_2$ vertices (respectively). By~\cref{cColourTW}, each $H_i$ is $2$-colourable with clustering  $O\parens{\sqrt{n_i}}$.
Let $n:=|V(H_1\boxtimes H_2)|=n_1n_2$. By using the product colouring as in the proof of~\cref{thm:UpperNoMaxDegree}, $H_1\boxtimes H_2$ is $4$-colourable with clustering  $O\parens{\sqrt{n}}$. Therefore, to prove that the upper bound of~\cref{thm:UpperNoMaxDegree} is asymptotically tight when $c=4$, one must choose $H_1$ and $H_2$ to be graphs such that every $2$-colouring of $H_i$ has clustering $\Theta(\sqrt{n_i})$, which is as large as possible. Fans have this property by~\cref{2ColouringFan}. This makes fans the natural candidates to show that~\cref{thm:UpperNoMaxDegree} is asymptotically tight. However, we now show that  $F_n\boxtimes F_n$ is 4-colourable with clustering $O( n^{2/3})$, and this bound is tight. 

\begin{prop}\label{FanFan4Colouring} 
There exists an absolute constant $\delta>0$ such that for all $n\geq 1$, every $4$-colouring of $F_n\boxtimes F_n$ has clustering at least $\delta n^{2/3}$.
\end{prop}

\begin{proof}
Let $\delta := (48)^{-2/3}$. We make no attempt to optimise this constant. 

Let $A$ and $B$ be copies of $F_k$, with $V(A)=\{a_0,a_1,\dots,a_n\}$ and $V(B)=\{b_0,b_1,\dots,b_n\}$, where $a_0$ dominates the path $a_1,\dots,a_n$ in $A$, and $b_0$ dominates the path $b_1,\dots,b_n$ in $B$. Let $A':=\{(a_1,b_0),\dots,(a_n,b_0)\}$ and $B':=\{(a_0,b_1),\dots,(a_0,b_n)\}$. 

Suppose, for the sake of contradiction, that $A\boxtimes B$ has a $4$-colouring with clustering at most~$\delta n^{2/3}$. Let the colours be red, green, blue, and black, where $(a_0,b_0)$ is red. A colour~$c$ is \defn{$A'$-empty} if no vertex in $A'$ is coloured~$c$, otherwise~$c$ is \defn{$A'$-nonempty}. Similarly, a colour~$c$ is \defn{$B'$-empty} if no vertex in $B'$ is coloured~$c$, otherwise~$c$ is \defn{$B'$-nonempty}.  A colour~$c$ is \defn{$A'$-small} if at most $\delta n^{2/3}$ vertices in $A'$ are coloured~$c$, otherwise~$c$ is \defn{$A'$-big}. Similarly, a colour~$c$ is \defn{$B'$-small} if at most $\delta n^{2/3}$ vertices in $B'$ are coloured~$c$, otherwise~$c$ is \defn{$B'$-big}. 

Since $(a_0,b_0)$ is dominant in $G$, there are at most $\delta n^{2/3}$ red vertices in total, so red is both $A'$-small and $B'$-small. If some colour $c$ is both $A'$-big and $B'$-nonempty, then all the vertices in $A'\cup B'$ coloured $c$ are in a single monochromatic component, which is a contradiction. So no colour is both $A'$-big and $B'$-nonempty. Similarly, no colour is both $A'$-nonempty and $B'$-big. Since $n>4\delta n^{2/3}$, at least one colour is $A'$-big, and at least one colour is $B'$-big. 

If at least two colours are $A'$-big and at least two colours are $B'$-big, then some colour is both $A'$-big and $B'$-big, which is a contradiction. Thus, without loss of generality, exactly one colour is $B'$-big. We may assume that black is $B'$-big and $A'$-empty, blue is $B'$-small, and green is $B'$-small. 

Let $H$ be the subgraph of $G$ induced by $\{a_1,\dots,a_n\}\times\{b_1,\dots,b_n\}$. So $H\cong P_n\boxtimes P_n$. 

Let $S$ be the union of:
\begin{itemize}
    \item the set $S_b$ of black vertices in $H$, 
    \item the set $S_r$ of red vertices in $H$, 
    \item the set $S_{r+}$ of vertices in $H$ (of any colour) adjacent to a red vertex in $A'\cup B'$, 
    \item the set $S_\ell$ of blue vertices in $H$ adjacent to a blue vertex in $A'\cup B'$,
    \item the set $S_g$ of green vertices in $H$ adjacent to a green vertex in $A'\cup B'$.
\end{itemize}

We claim that $|S_b|\leq 4\delta n^{5/3}$. 
Let $I:=\{j\in\{1,\dots,n\}:(a_0,b_j)\text{ is black}\}$. For each $j\in I$ there are at most $\delta n^{2/3}$ black vertices $(a_i,b_j)$ (since they are dominated by $(a_0,b_j)$). Thus $|S_b|\leq |I|\,\delta n^{2/3}+(n-|I|)n$. Since each of red, blue and green is $B'$-small, $n-3\delta n^{2/3}\leq |I|\leq n$ implying $|S_b|\leq 4\delta n^{5/3}$, as claimed. 

We have $|S_r|<\delta n^{2/3}$ directly. Each red vertex in $A'\cup B'$ contributes at most $3n$ vertices to $S_{r+}$, so $|S_{r+}|\leq 3\delta n^{5/3}$. 

We claim that $|S_\ell|\leq 2\delta n^{5/3}$. Each vertex in $S_\ell$ is adjacent to a blue vertex in $A'\cup B'$. Each blue vertex in $A'\cup B'$ contributes at most $\delta n^{2/3}$ vertices to $S_\ell$. So $|S_\ell|\leq 
|A'\cup B'| \delta n^{2/3} \leq 2\delta n^{5/3}$. 
By the same argument, $|S_g|\leq 2\delta n^{5/3}$. 

In total, $|S|\leq 4\delta n^{5/3} + \delta n^{2/3} + 3\delta n^{5/3} + 2\delta n^{5/3} + 2\delta n^{5/3}< 12 \delta n^{5/3}$. 

Every vertex in $H-S$ is blue or green. 
Suppose, for the sake of contradiction, that $vw$ is an edge of $H-S$ with $v$ blue and $w$ green. 
Say $v=(a_i,b_j)$ and $w=(a_{i'},b_{j'})$. 
Since $vw$ is an edge of $H-S$, we have $|i-i'|\leq 1$ (and $|j-j'|\leq 1$). 
Thus $v$ and $w$ are both adjacent to $(a_i,b_0)$ and $(a_{i'},b_0)$. 
Since $v\not\in S_\ell$, neither 
$(a_i,b_0)$ nor  $(a_{i'},b_0)$ is blue.
Since $w\not\in S_g$, neither 
$(a_i,b_0)$ nor  $(a_{i'},b_0)$ is green.
Since black is $A'$-empty, 
both $(a_i,b_0)$ and  $(a_{i'},b_0)$ are red, 
which implies that $v$ and $w$ are in $S_{r+}$, giving a contradiction. 
Hence, there is no blue--green edge in $H-S$.
This implies that the components of $H-S$ are monochromatic, so each component of $H-S$ has at most $\delta n^{2/3}$ vertices. 
Now by \cref{CompGrid} with $k=\delta n^{2/3}$, 
$$n^2\leq 4|S|k^{1/2} <
4\cdot 12 \delta n^{5/3}\, (\delta n^{2/3})^{1/2}
= 48 \,\delta^{3/2} n^2 \leq n^2,$$
which is a contradiction.
\end{proof}

The clustering value in~\cref{FanFan4Colouring} is again best possible (up to the multiplicative constant), as demonstrated by the following result.

\begin{prop}\label{prop:counterexampleFans}
    For any integer $n\geq 2$, there exists a $4$-colouring of $F_{n^3}\boxtimes F_{n^3}$ with clustering at most $7n^2$.
\end{prop}

\begin{proof}
Let $G:=F_{n^3}\boxtimes F_{n^3}$. We label the vertices of $G$ as $\set*{(i,j)\colon 0\leq i,j\leq n^3}$, where $(0,0)$ is dominating for $G$, and for all $i\in[n^3]$, $(i,0)$ is adjacent with all vertices in $\set*{(i,j)\colon j\in[n^3]}$, and $(0,i)$ is adjacent with all vertices in $\set*{(j,i)\colon j\in[n^3]}$. Then the vertices $\set{(i,j)\colon 1\leq i,j\leq n^3}$ form a copy of $P_{n^3}\boxtimes P_{n^3}$ in $G$.  

Start by colouring $(0,0)$ black. Let $X=\{(0,i)\colon i\in[n^3]\}$ and $Y=\{(i,0)\colon i\in[n^3]\}$, as illustrated in~\cref{fig:TikzcounterexampleFans}. The vertices of $X$ induce a path of length $n^3$. Colour almost all vertices of $X$ blue, inserting three consecutive black vertices after every $n$ blue vertices. Formally, for every integer $\ell\in\{1,\ldots,n^2\}$, colour $(0,\ell n -1)$, $(0,\ell n)$ and $(0,\ell n+1)$ black, and the rest of $X$  blue. Similarly, colour $Y$ red, inserting black vertices. Note that there are at most $7n^{2}$ black vertices. 

\begin{figure}[ht]
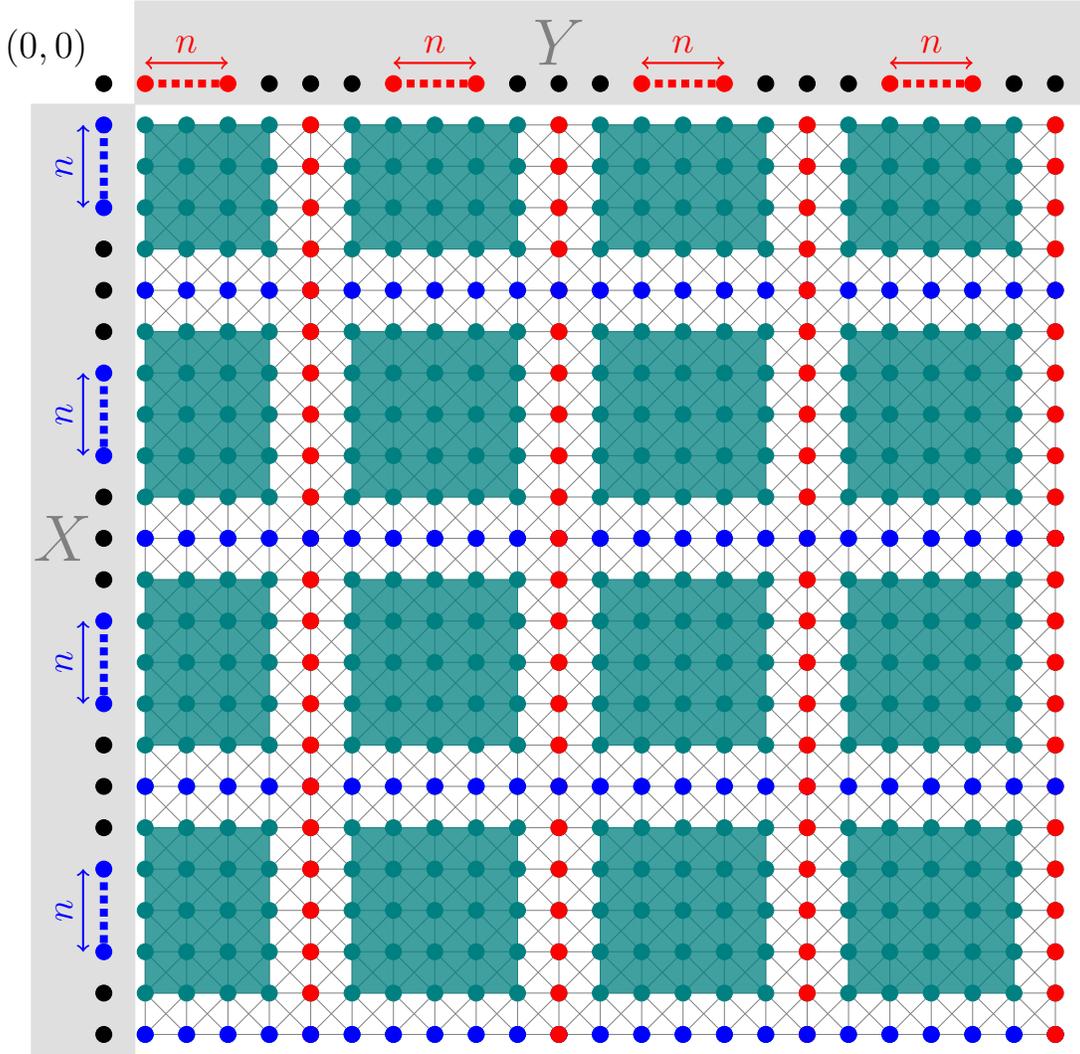

    \centering
    \FanFan{0.55}
    \caption{A $4$-colouring of $F_{n^3}\boxtimes F_{n^3}$ with clustering at most $7n^2$.}
    \label{fig:TikzcounterexampleFans}
\end{figure}    

For every $\ell\in\{1,\dots,n^{2}\}$ and $i,j\in\{1,\dots,n^3\}$, colour vertices of the form $(\ell n,j)$, $(i,\ell n)$ as follows. Fix $\ell\in\{1,\dots,n^{2}\}$. For every $j\in\{1,\dots,n^3\}$, colour $(\ell n,j)$ blue, except vertices $(\ell n,\ell' n)$ where $0<\ell'\leq n^{2}$ and $\ell\neq\ell'\pmod 2$. Colour all other vertices of the form $(\ell n,j)$, $(i,\ell n)$ red.

Finally, colour all remaining vertices green. Note that there is no red vertex in red-dominated columns (or adjacent to one), and no blue vertex in blue-dominated rows (or adjacent to one). The largest blue component has size $2n+5$, likewise for red, and the largest green component has size $(n+2)^2 < 7 n^{2}$. 
\end{proof}


\section{Products of bounded degree graphs}
\label{sec:BothBoundedDegree}

This section considers clustered colouring of the strong product of two graphs, both with bounded treewidth and bounded degree. This setting is much simpler than with none or one of the graphs having bounded degree. We only require the following proposition to solve this case.

\begin{prop}
\label{thm:project}
For any fixed positive integer $k$ and any graphs $H_1$, $H_2$ that are $c$-colourable with clustering $k$, the graph $H_1\boxtimes H_2$ is $c$-colourable with clustering $k\,|V(H_1\boxtimes H_2)|^{1/2}$. 
\end{prop}

\begin{proof} Let $G:=H_1\boxtimes H_2$ and $n:=|V(G)|$. Without loss of generality, $|V(H_2)|\leq |V(H_1)|$, so $|V(H_2)|\leq n^{1/2}$. 
Fix a $c$-colouring of $H_1$ with clustering $k$. Colour each vertex $(x,y)$ of~$G$ by the colour assigned to $x$. Now $G$ is $c$-coloured. If $(x_1,y_1)(x_2,y_2)$ is a monochromatic edge in $G$, then $x_1$ and $x_2$ had the same colour in $H_1$, and $x_1=x_2$ or $x_1x_2\in E(H_1)$. Hence, each monochromatic component $M$ in $G$ projects to a monochromatic component in $H_1$, implying that $|V(M)|\leq k\,|V(H_2)|\leq k\,n^{1/2}$, as claimed. 
\end{proof} 

Now consider graphs $H_1,H_2$ with treewidth at most $t$ and maximum degree at most~$\Delta$. By~\cref{lem:containedtree}, each $H_i$ is contained in $T_i\boxtimes K_{O(t\Delta)}$ for some tree $T_i$. As noted by \citet[Theorem~2.2]{ADOV03}, any proper 2-colouring of $T_i$ determines a $2$-colouring of $H_i$ with clustering $O(t\Delta)$. Therefore:

\begin{itemize}
    \item For $c=2$, \cref{thm:project} implies that $H_1\boxtimes H_2$ is 2-colourable with clustering $O(t\Delta\,|V(H_1\boxtimes H_2)|^{1/2})$. This bound is tight for fixed $t$ and $\Delta$, since the Hex Lemma (\cref{lemma:HEX}) implies that any $2$-colouring of $P_{n}\boxtimes P_{n}$ has clustering at least~$n=|V(P_n\boxtimes P_n)|^{1/2}$.
    \item For $c=3$, \citet[Theorems 4 and 33]{EW23} proved that $H_1\boxtimes H_2$ is $3$-colourable with bounded clustering, at most $O(t^3\Delta^4)$.
    \item For $c=4$, using the product of the $2$-colourings of each $H_i$, it is easy to see that $H_1\boxtimes H_2$ is $4$-colourable with bounded clustering, at most $O(t^2\Delta^2)$. 
\end{itemize}


\section{Planar graphs}
\label{Planar}

Now consider clustered colouring of planar graphs.  \citet{LMST08} introduced the following definition. Let $f_c(n)$ be the minimum integer such that every $n$-vertex planar graph is $c$-colourable with clustering $f_c(n)$. The 4-Colour Theorem says that $f_c(n)=1$ for every $c\geq 4$.  \citet{LMST08} showed that $f_2(n)=\Theta(n^{2/3})$ and $\Omega(n^{1/3}) \leq f_3(n) \leq O(n^{1/2})$. 

One approach for closing the gap in these bounds on $f_3(n)$ is to apply the  Planar Graph Product Structure Theorem of \citet*{DJMMUW20}, which says that every planar graph is contained in $H\boxtimes P$ for some graph $H$ of treewidth at most $8$ and some path $P$. The idea is to 3-colour $H\boxtimes P$ with small clustering, thus determining a 3-colouring of $G$. Two issues arise, however. First, it may be that $|V(H\boxtimes P)|$ is significantly larger than $|V(G)|$ (it is only known that $|V(H)|\leq|V(G)|$ and $|V(P)|\leq |V(G)|$), so a clustering function of $O(|V(H\boxtimes P)|^\beta)$ does not imply a clustering function of $O(|V(G)|^\beta)$. Second, by \cref{thm:3lowerbound}, there are graphs $H$ with treewidth 3 and there are paths $P$ such that every 3-colouring of $H\boxtimes P$ has clustering $\Omega(|V(H\boxtimes P)|^{3/7})$. So the best 
upper bound on $f_3(n)$ that one could hope for using this method is $f_3(n)\leq O(n^{3/7})$, which is between the known bounds mentioned above. Determining $f_3(n)$ is a tantalising open problem.

\subsection*{Acknowledgements} This research was initiated at the \href{https://www.matrix-inst.org.au/events/structural-graph-theory-downunder-ll/}{Structural Graph Theory Downunder II} program of the Mathematical Research Institute MATRIX (March 2022).

{
\fontsize{10pt}{11pt}\selectfont
\bibliographystyle{DavidNatbibStyle}
\bibliography{DavidBibliography}
}
\end{document}